\documentclass{article}
\pdfoutput=1
\usepackage[a4paper,top=2cm,bottom=2cm,left=3cm,right=3cm,marginparwidth=1.75cm]{geometry}

\usepackage[utf8]{inputenc}
\usepackage{amsfonts}
\usepackage{amsmath,amscd,amsthm,amssymb,mathrsfs,setspace}
\usepackage{mathtools}
\usepackage{cite}
\usepackage{color}
\usepackage{algorithm}
\usepackage{algorithmicx}
\usepackage{algpseudocode}
\usepackage{booktabs}
\usepackage{multirow}
\usepackage{makecell}
\usepackage[colorlinks]{hyperref}
\usepackage[nameinlink,capitalize]{cleveref}
\usepackage{graphicx}
\usepackage{tikz}
\usepackage{url}
\usepackage{fancyhdr}
\usepackage{authblk}
\usepackage{enumitem}
\usepackage{adjustbox}
\usepackage{subcaption}
\usepackage{xcolor, soul}
\usepackage{comment}
\usepackage{adjustbox}
\usepackage{nicefrac}
\sethlcolor{pink}
\usetikzlibrary{patterns,patterns.meta}
\definecolor{lccx}{HTML}{92268F}

\ifpdf
\hypersetup{
	pdftitle={Convergence of trust-region algorithms in metric spaces},
	pdfauthor={Paul Manns},
	linkcolor=lccx	
}

\newtheorem{theorem}{Theorem}[section]
\newtheorem{proposition}[theorem]{Proposition}
\newtheorem{lemma}[theorem]{Lemma}
\newtheorem{corollary}[theorem]{Corollary}

\newtheorem{assumption}[theorem]{Assumption}

\theoremstyle{remark}
\newtheorem{remark}[theorem]{Remark}

\crefname{theorem}{Theorem}{Theorems}
\Crefname{theorem}{Theorem}{Theorems}
\crefname{assumption}{Assumption}{Assumptions}
\Crefname{assumption}{Assumption}{Assumptions}
\crefname{lemma}{Lemma}{Lemmas}
\Crefname{lemma}{Lemma}{Lemmas}
\crefname{definition}{Definition}{Definitions}
\Crefname{definition}{Definition}{Definitions}
\crefname{proposition}{Proposition}{Propositions}
\Crefname{proposition}{Proposition}{Propositions}
\crefname{algorithm}{Algorithm}{Algorithms}
\Crefname{algorithm}{Algorithm}{Algorithms}
\crefname{section}{Section}{Sections}
\Crefname{section}{Section}{Sections}
\crefname{appendix}{Appendix}{Appendices}
\Crefname{appendix}{Appendix}{Appendices}

\DeclareMathOperator*{\argmax}{arg\,max}

\DeclareMathOperator*{\dist}{dist}

\DeclareMathOperator*{\TV}{TV}

\DeclareMathOperator*{\pred}{pred}
\DeclareMathOperator*{\ared}{ared}

\DeclareMathOperator*{\BV}{BV}
\DeclareMathOperator*{\BVW}{BV_W}

\DeclareMathOperator*{\supp}{supp}

\newcommand{\N}{\mathbb{N}}
\newcommand{\R}{\mathbb{R}}
\newcommand{\Z}{\mathbb{Z}}
\newcommand{\Ha}{\mathcal{H}}

\newcommand{\weakstarto}{\stackrel{\ast}{\rightharpoonup}}
\newcommand{\crit}{\mathcal{C}}
\newcommand{\calM}{\mathcal{M}}

\newcommand*\dd{\mathop{}\!\mathrm{d}}

\crefname{assumption}{Assumption}{Assumptions}
\Crefname{assumption}{Assumption}{Assumptions}
\crefname{lemma}{Lemma}{Lemmas}
\Crefname{lemma}{Lemma}{Lemmas}
\crefname{definition}{Definition}{Definitions}
\Crefname{definition}{Definition}{Definitions}
\crefname{proposition}{Proposition}{Propositions}
\Crefname{proposition}{Proposition}{Propositions}
\crefname{algorithm}{Algorithm}{Algorithms}
\Crefname{algorithm}{Algorithm}{Algorithms}

\title{Convergence of trust-region algorithms in metric spaces}

\author[1]{Paul Manns}
\affil[1]{Faculty of Mathematics, TU Dortmund University, 44227 Dortmund, Germany, \texttt{paul.manns@tu-dortmund.de}}

\begin{document}
\maketitle
\begin{abstract}
Trust-region algorithms can be applied to very abstract optimization problems
because they do not require a specific direction of descent or gradient.
This has lead to recent interest in them, in particular in the area of integer
optimal control problems, where the infinite-dimensional problem
formulations do not assume vector space structure.

We analyze a trust-region algorithm in the abstract setting of a metric space, a setting in
which integer optimal control problems with total variation regularization can be formulated.
Our analysis avoids a reset of the trust-region radius upon acceptance of the iterates when proving
convergence to stationary points. This reset has been present
in previous analyses of trust-region algorithms for integer optimal control
problems. Our computational benchmark shows that the runtime can be
considerably improved when avoiding this reset, which is now theoretically
justified.
\end{abstract}
\section{Introduction}
In this work, we analyze a trust-region algorithm for optimization 
problems of the form
\begin{gather}\label{eq:p_abstract}
\min_{x \in X} J(x),\tag{P}
\end{gather}
where $(X,d)$ denotes a metric space $X$ with a metric $d$. Trust-region algorithms
are particularly suited for such abstract settings because
they only require a model function and an algorithm
(oracle) that optimizes over this model function in a trust region. In particular, there is no need for a specific \emph{direction of steepest descent}, which is difficult to define
without assuming vector space structure. There has not been
much work on analyzing trust-region algorithms
without vector space structure, where we are aware of the works
\cite{leyffer2022sequential,manns2023on,hahn2023binary,manns2023convergence,sharma2021inversion,terpin2022trust}. Importantly, none of them
explicitly address this lack of structure and 
provide a comprehensive trust-region algorithm analysis
in an abstract setting.

The special cases that motivate us and will provide intuitive guidance through the
article are optimization problems of the form
\begin{gather}\label{eq:q}
\min_{w \in L^2(0,1)} J(w) \coloneqq F(w) + \TV(w)
\quad\text{s.t.}\quad
w(x) \in W \text{ a.e.}\tag{Q}
\end{gather}
with $W = \{w_1,\ldots,w_M\} \subset \Z$, where we assume $w_1 < \ldots < w_M$.
Here, $F$ is the main part of the objective that may contain the
solution operator of a differential equation. The term
$\TV(w) \in [0,\infty]$ denotes the total variation
of $w$, which is the sum of the jump heights of the function $w$.
The finiteness of $W$ implies that the height of a jump
of $w$ is bounded below, which is used later in the analysis of stationarity.
The assumption $W \subset \Z$ is not important for our analysis but keeps \eqref{eq:q}
consistent with the referenced literature. Consequently, such optimization problems
can be categorized as \emph{integer optimal control problems with switching costs} as have, for example, been considered in the
control community \cite{de2019mixed,bestehorn2021mixed,kirches2019generation,gau2017novel} more intensively in recent years. In the situation of \eqref{eq:q}, $(X,d)$ arises as follows if $F$ is bounded below. There is a uniform upper 
bound on $\TV(w_n)$ for iterates $w_n$ produced by a descent 
algorithm, which implies that we can wlog replace the feasible set by
\begin{gather}\label{eq:X_for_Q}
X \coloneqq \{ w \in \BV(0,1)\,:\, w(x) \in W\text{ a.e.\ and } \TV(w) \le F(w_0) + \TV(w_0) - \inf F \},
\end{gather}
where $w_0$ is the initial iterate of the algorithm. Then, we
can choose
\[ d(v,w) \coloneqq \|v - w\|_{L^1}, \]
which is a metric on this space. Because $X$ is uniformly bounded with respect to $\TV$ and $\|\cdot\|_{L^1}$
(note that $W$ is finite), $(X,d)$ is a weakly$^*$ sequentially compact subset of $L^1(0,1)$.
Note that other $L^p$-norms are possible here as well. The finiteness of $W$ implies that
sequences in $X$ that converge in $L^p(0,1)$, $p \ge 1$, also converge in $L^q(0,1)$, $q > p$. The
$\|\cdot\|_{L^1}$-norm is often beneficial here because it has a straightforward
polyhedral description after discretization.

We briefly provide motivation for improved analysis of trust-region
algorithms for \eqref{eq:q}. Among several approaches that exist for tackling problems of the form
\eqref{eq:q} and closely related ones
\cite{buchheim2024parabolic,sager2021mixed,bestehorn2021mixed,bemporad2002master,burger2023gauss},
trust-region algorithms have gained attention recently because they can be meaningfully defined
as descent algorithms for such problems \cite{leyffer2022sequential,hahn2023binary,manns2023convergence,marko2023integer}.
Due to the properties of the $\TV$-seminorm, it is possible to define a meaningful stationarity (first-order optimality)
concept and find stationary points of \eqref{eq:q} by means of a trust-region algorithm \cite{leyffer2022sequential},
which can be used to improve heuristically found feasible points or to compute upper bounds in branch-and-bound approaches like
\cite{buchheim2024parabolic} if \eqref{eq:q} is optimized to global optimality.

The solutions to the trust-region subproblems arising from \eqref{eq:q} can be computed by means of
efficient combinatorial algorithms \cite{severitt2023efficient,marko2023integer}. The convergence analysis of the trust-region algorithm
from \cite{leyffer2022sequential} has been tailored to \eqref{eq:q} and has the drawback that it resets the
trust-region radius in successful iterations, that is, if a computed step is accepted. This may hamper performance in 
practice because subproblems with larger trust-region radii generally have longer compute times than those with small 
trust-region radii and the reset can induce a large number of unsuccessful and unnecessary step computations because the 
reset radius is \emph{just too large}. Moreover, the question if this reset can be dropped without losing guarantees on
the asymptotics of the algorithm has been posed to the author at conferences.

\paragraph{Contribution}
In this work, we provide a convergence proof for a trust-region algorithm in the abstract setting of metric spaces introduced above in \eqref{eq:p_abstract},
where we are able to obtain the typically desired convergence property, namely that accumulation points are stationary. Importantly, we are able to do so
without requiring a reset of the trust-region radius so that we can answer the aforementioned question regarding \eqref{eq:q} positively. In the algorithm
analysis and the problem assumptions, we replace the typically required continuity of a criticality measure by assuming lower semi-continuity of the
criticality measure and a compensation of discontinuities in the criticality measure by means of a corresponding lower bound on the predicted 
reduction for small enough trust-region radii.

To show that this analysis is practically valuable, we verify all imposed assumptions for a class of instances of \eqref{eq:q}.
Moreover, we execute the trust-region algorithm for benchmark instances of \eqref{eq:q} with and without trust-region radius
reset to assess the practical performance impact of our theoretical advancement with respect to \eqref{eq:q}. Avoiding the reset of the trust-region radius
leads to a substantial improvement in runtime (more than 50\,\% on more
than half of our instances) while negatively impacting the quality of the
returned points (objective values of the last accepted iterates) of the algorithm to a lesser extent.

\paragraph{Structure of the remainder}
We first introduce some notation below.
In contrast to the analysis in \cite{leyffer2022sequential}, we base our proof on so-called criticality measures and we
briefly explain how they typically work and why we required a different proof strategy in \cref{sec:changes}.
In \cref{sec:problem} we provide the algorithm and the criticality measure for our model problem class \eqref{eq:q}
to give an intuition which terms can be behind the criticality measure and the predicted reduction in an abstract
setting before starting with the abstract problem analysis.
In \cref{sec:abstract} we consider the abstract setting and variant of the algorithm as well as appropriate assumptions
and prove the convergence of the algorithm under said assumptions. In \cref{sec:verification}, we verify the
assumptions for \eqref{eq:q}. \cref{sec:numerics} contains our computational results and we draw a brief conclusion in \cref{sec:conclusion}.

\paragraph{Notation} We denote the space of $\R^m$-valued Radon measures on a bounded Lipschitz domain $\Omega$ by $\calM(\Omega,\R^m)$
or short $\calM$ if unambiguous from the context. We denote the variation measure of $\mu \in \calM(\Omega,\R^m)$ by $|\mu|$. Similarly,
the function that is given by the (pointwise a.e.) absolute value or norm of $f \in L^1(\Omega,\R^m)$ is denoted by $|f|$ or $\|f\|$
(if $m \ge 2$). We denote the continuous bilinear form associated with the duality of$C_0(\Omega,\R^m)$, the space of continuous functions
that vanish at the boundary of $\Omega$, and $\calM(\Omega,\R^m)$, the space of Radon measures on $\Omega$ by
$\langle \cdot, \cdot\rangle_{\calM, C}$; see \cite[Thm 1.54]{ambrosio2000functions}. Thus, a sequence of Radon measures
$\{\mu_n\}_n \subset \calM(\Omega,\R^m)$ converges weakly$^*$ to some limit $\mu$ if and only if
$\langle \mu_n, \phi\rangle_{\calM, C} \to \langle \mu, \phi\rangle_{\calM, C}$ for all $\phi \in C_0(\Omega,\R^m)$, which we denote by
$\mu_n \weakstarto \mu$. For a differentiable function $F : L^p(\Omega) \to \R$, $p \in [1,\infty)$ with H\"older
conjugate index $p'$, $F'(x)$ has a representative in $L^{p'}(\Omega)$, which we denote by $\nabla F(x)$.
A sequence $\{x_n\}_n$ converges weakly$^*$ in $\BV(\Omega)$ to some limit $x \in \BV(\Omega)$
if and only if $x_n \to x$ in $L^1(\Omega)$ and $\sup_{n} \TV(x_n) < \infty$, which we denote again by
$x_n \weakstarto x$; see \cite[Prop.\ 3.13]{ambrosio2000functions}.
The continuous embedding of a Banach space $X$ into a Banach space $Y$ is denoted by $X \hookrightarrow Y$.

\section{Trust-region algorithm analysis with criticality measures}\label{sec:changes}
A standard tool in the convergence analysis of trust-region algorithms are criticality measures
that are also known as gap functions \cite{larsson1994class} in the literature.
A sensible criticality measure is a non-negative function $\crit : X \to [0,\infty)$ that satisfies $\crit(x) = 0$
if and only if $x$ is stationary, that is if $x$ is feasible and satisfies a (first-order) optimality condition for
\eqref{eq:p_abstract}.

Typical convergence proofs of trust-region algorithms that are based on such criticality measures enforce sufficient decrease
in every iteration, which gives
\[ \liminf_{n\to\infty} \crit(x_n) = 0 \]
over the iterations, see, for example, \cite{toint1997non,manns2023convergence}. Then one typically proceeds with a contradictory
argument, where it is assumed that there exists $\varepsilon > 0$ such that
\[
\limsup_{n\to\infty} \crit(x_n) > \varepsilon > \liminf_{n\to\infty} \crit(x_n) = 0.
\]
Then one shows
that there must exist a Cauchy sequence of iterates $\{x_k\}_k$ such that
\[ \lim_{\ell \to \infty} \crit(x_{k_\ell}) = \varepsilon
\quad\text{and}\quad
\lim_{m\to\infty} \crit(x_{k_m}) = 0
\]
holds for subsequences $\{x_{k_\ell}\}_\ell$, $\{x_{k_m}\}_m$.
Because the whole sequence is Cauchy, this leads to a contradiction to the continuity of $\crit$ and thus proves
\[ \limsup_{n\to\infty} \crit(x_n) = 0. \]
This continuity of $\crit$ also implies that $\crit(\bar{x}) = 0$ holds for all accumulation points $\bar{x}$ of $\{x_n\}_n$,
that is, every accumulation point produced by the algorithm is stationary if the feasible set is closed.

For our guiding example \eqref{eq:q}, it is possible
to define a criticality measure for the stationarity concept
from \cite{leyffer2022sequential}. This criticality measure is not continuous with respect to weak$^*$ convergence in
$\BV(0,1)$, however, implying that a typical argument as sketched above and, for example, presented in and below (A.29)
in \cite{toint1997non} fails. This insight led to the approach in \cite{leyffer2022sequential} with the aforementioned
reset of the trust-region radius in successful iterations. Note that such trust-region reset ideas have also been used
in the analysis of sequential quadratic programming algorithms with trust-region globalization \cite{fletcher2002global}
and for nonsmooth complementarity constraints \cite{kirches2022sequential}.

Here, we take an abstract point of view and consider criticality measures that are only
lower semi-continuous with respect to convergence in the metric $d$ of a metric space $(X,d)$.
We can show under several assumptions depending on $\crit$ and the trust-region subproblem
that the typically desired results
\[ \liminf_{n\to\infty} \crit(x_n) = 0\quad\text{and}\quad\lim_{n \to \infty}\crit(x_n) = 0 \]
can still be obtained, where the latter requires more assumptions on $\crit$ than the former.
Specifically, the latter result can be obtained if discontinuities in the criticality
measure are compensated by a lower bound on the predicted reduction for small enough
trust-region radii. Together with the lower semi-continuity of $\crit$, all accumulation points
are stationary. If $(X,d)$ is a compact space as well, accumulation points necessarily exist.

We verify all of the imposed assumptions for a sensible setting of
\eqref{eq:q} in \cref{sec:verification}. Since the criticality measure is also lower semi-continuous wrt.\
weak$^*$ convergence in $\BV(0,1)$ on the weak$^*$ compact set $X$ from \eqref{eq:X_for_Q},
all accumulation points are stationary and at least one exists.

\section{Criticality measure and algorithm specification for \eqref{eq:q}}\label{sec:problem}
We define $\BVW(0,1) \coloneqq \{ w \in \BV(0,1)\,:\, w(x) \in W \text{ a.e.}\}$.
We first make the following assumption on our problem.
\begin{assumption}\label{ass:general_assumption}
Let $X \subset \BVW(0,1)$ be a bounded set with respect to
$\TV$. We make the following assumptions for a given and fixed $p \ge 1$.
\begin{enumerate}
\item Let $F : L^1(\Omega) \to \R$ be bounded below.
\item Let $F : L^1(\Omega) \to \R$ be differentiable such that
\begin{itemize}
\item $\nabla F : L^1(\Omega) \to L^\infty(\Omega)$ is Lipschitz continuous,
\item $\{\nabla F(w) \,:\, w \in X \}$ is uniformly bounded
in $W^{1,p}(\Omega)$.
\end{itemize}
\end{enumerate}
\end{assumption}
\begin{remark}
There is a slight difference between \cref{ass:general_assumption} 
and the assumptions in \cite{leyffer2022sequential}. Specifically, 
\cref{ass:general_assumption} 2.\ replaces the twice differentiability
of $F$ including uniform boundedness of the Hessian form with respect to the
product of the $L^1$-norms of the arguments (Assumption 4.1 in \cite{leyffer2022sequential}).
While similar in nature, the differentiability requirement in \cite{leyffer2022sequential}
is higher but requires less structural assumptions on $\nabla F(x)$ for $x \in X$.
We find the setting of \cref{ass:general_assumption} sensible and instructive because 
it gives a broad setting for which convergence of a trust-region algorithm can
be verified in \cref{sec:verification} that at the same time does not hinge on a 
second-order analysis.
\end{remark}

Under \cref{ass:general_assumption}, the following first-order optimality condition can be shown for
\eqref{eq:q}. Since the measure-valued distributional derivative $D\bar{w}$
of $\bar{w}$ is a sum of Dirac measures, it means that
$\nabla F(w)$ is zero at the (finitely many) jump points of $w$.
 With a slight abuse of notation, we will consider duality pairings 
\[ \langle |Dw|, g \rangle_{\calM,C}, \]
where $w \in \BVW(0,1)$ and $g \in C([0,1])$. This is sensible because $w \in \BVW(0,1)$
gives that $Dw$ is a weighted sum of Dirac measures that are located on the jump set of $w$
with the weights being the jump heights. The weights are bounded away from zero because $W$ is finite
and thus $\TV(w) < \infty$ implies that $Dw$ is concentrated on finitely many points that are
located strictly inside $(0,1)$ so that $g$ can be changed near the boundary to obtain a function in
$C_0(0,1)$ without changing $g$ on the support of $Dw$ or $|Dw|$.
\begin{proposition}\label{prp:stationarity}
Let $\cref{ass:general_assumption}$ hold for arbitrary subsets $X \subset \BVW(0,1)$ that are bounded
with respect to $\TV$ and any $p \ge 1$. Let $\bar{w} \in X$ be a local minimizer of \eqref{eq:q},
that is there exists $r > 0$ such that
\[ F(\bar{w}) + \TV(\bar{w}) \le F(w) + \TV(w)
   \text{ for all } w \in \BVW(0,1) \text{ with } \|w - \bar{w}\|_{L^1(0,1)} \le r.
\]
Then
\begin{gather}\label{eq:stationarity}
\langle |D\bar{w}|, |\nabla F(\bar{w})| \rangle_{\calM, C} = 0.
\end{gather}
\end{proposition}
\begin{proof}
This follows from Lemma 4.10 in \cite{leyffer2022sequential} together with the aforementioned fact that $Dw$ is a
weighted sum of Dirac measures that are located on the jump set of $w$, which is again
finite because $\TV(w) < \infty$ holds. Then the continuity of $\nabla F(w)$ follows from the
continuous embedding $W^{1,p}(0,1) \hookrightarrow C([0,1])$ and we apply the aforementioned
implicit modification when evaluating the duality pairing.
Note that the proof of Lemma 4.10 in \cite{leyffer2022sequential}
makes an assumption on the Hessian
of $F$ but a close inspection shows that the assumed Lipschitz continuity
$\nabla F : L^1(0,1) \to L^\infty(0,1)$ together with the mean value
theorem also yield the claim therein.
\end{proof}
Consequently, we will refer to feasible points of \eqref{eq:q} that satisfy \eqref{eq:stationarity}
as stationary points; see also \cite[Prop.\ 4.17]{leyffer2022sequential}.
We now provide a trust-region algorithm and its ingredients that differs
from the one presented in \cite{leyffer2022sequential} in that
it does not reset the trust-region radius when the iteration is
successful, that is when
a step is accepted.
Instead, the trust-region radius is doubled in this case.

We start by defining the trust-region subproblem as in Section 3.1 in \cite{leyffer2022sequential} below:
\begin{gather}\label{eq:tr}
\text{{\ref{eq:tr}}}(\bar{w}, g, \Delta) \coloneqq
\left\{
\begin{aligned}
\min_{w \in L^2(\Omega)}\ & (g, w - \bar{w})_{L^2} + \TV(w) - \TV(\bar{w})\\
\text{s.t.}\quad & \|w - \bar{w}\|_{L^1} \le \Delta\text{ and }w(x) \in W
\text{ for a.e.\ } x \in \Omega.
\end{aligned}
\right.
\tag{TR}
\end{gather}
The function $g$ in the trust-region subproblem will be the gradient of $F$ and consequently, this trust-region subproblem
uses a linear model of the smooth part of the objective and keeps the nonsmooth term $\TV$ exactly.
The convergence analysis below abstracts from the structure of the trust-region subproblem and can be carried out for higher-order models of $F$ too if the regularity assumptions on
$F$ carry over to the employed model. Solving discretized instances of 
\eqref{eq:tr} will likely become more involved since they are no longer
integer linear programs in this case and it is unclear if the efficient algorithms from
\cite{severitt2023efficient} can be transferred even if the higher-order terms
in the model are convex.
After solving a trust-region subproblem, predicted and actual reduction arecomputed in trust-region algorithms. To this end,
the reduction of the model function and the reduction of the true objective function for the computed (approximate)
solution to the trust-region subproblem are evaluated. Thus, they can be defined as:
\begin{align*}
\pred(\bar{w},\Delta) &= -\text{{\ref{eq:tr}}}(\bar{w}, \nabla F(\bar{w}), \Delta) \\
\ared(\bar{w},w) &= F(\bar{w}) + \TV(\bar{w}) - F(w) - \TV(w).
\end{align*}
In order to determine if the computed solution to the trust-region subproblem can be accepted as a new iterate
or not, trust-region algorithms check if the actual reduction is at larger than a fraction of the predicted
reduction, where the acceptance ratio $\sigma \in (0,1)$ is usually kept constant over the course of the algorithm.
The acceptance criterion reads:
\begin{gather}\label{eq:accept}
\ared(\bar{w},w) \ge \sigma \pred(\bar{w},\Delta).
\end{gather}
With these ingredients, a trust-region algorithm can be defined by solving the trust-region subproblem,
evaluating predicted and actual reduction, and checking if the step can be accepted. If yes, we
say that the iteration is
\emph{successful}, the computed solution is used as the next iterate and the trust-region radius is doubled
(the doubling occurs until a maximum trust-region radius $\Delta_{\max}$ is reached). If not,
the current iterate is kept as the next iterate and the trust-region radius is halved. This algorithm
is specified in \cref{alg:trm}.
\begin{algorithm}[t]
\caption{Trust-region algorithm to optimize \eqref{eq:q} based on Algorithm 1 in\cite{leyffer2022sequential}
without trust-region radius reset}\label{alg:trm}
\textbf{Input:} $F$, $\nabla F$, $\Delta_{\max} \in (0,\infty]$, $w_0 \in \BVW(\Omega)$, $\sigma \in (0,1)$.

\begin{algorithmic}[1]
\For{$n = 0,1,2\ldots$}
	\State $g_n \gets \nabla F(w_n)$
	\State $\tilde{w}_n \gets \arg \text{{\ref{eq:tr}}}(w_n, g_n, \Delta_n)$ \label{ln:tr}
	\State $\pred_{n} \gets \pred(w_n, g_n, \Delta_n)$
	\State $\ared_{n} \gets \ared(w_n,\tilde{w}_n)$
	\If{$\pred_n = 0$}
		\State Terminate and return $w_n$.
	\ElsIf{$\ared_{n} \ge \sigma \pred_{n}$}\label{ln:sufficient_decrease}
		\State $w_{n+1} \gets \tilde{w}_n$
		\State $\Delta_{n+1} \gets \min\{2\Delta_n,\Delta_{\max}\}$\label{ln:increase_tr}
	\Else
		\State $w_{n+1} \gets w_n$
		\State $\Delta_{n+1} \gets 0.5 \Delta_n$
	\EndIf
\EndFor
\end{algorithmic}
\end{algorithm}
There are a wealth of possible modifications to \cref{alg:trm} to improve performance in practice
and we refer to the book \cite{conn2000trust} for further reading. The important point in this work is that the
previous convergence analysis in \cite{leyffer2022sequential} required a reset of the trust-region radius in
successful iterations to show convergence to stationary points, which effectively means that \cref{alg:trm} 
\cref{ln:increase_tr} is replaced by $\Delta_{n+1} \gets \Delta_0$ in \cite{leyffer2022sequential}.

In the next section, we will analyze an abstract variant of \cref{alg:trm}. As is typical for trust-region algorithms,
see, e.g., \cite{toint1997non}, we will base the convergence analysis on a so-called \emph{criticality measure}
or \emph{gap function} that provides a means to quantify the non-stationarity. In the easiest unconstrained case,
the norm of the gradient can be used. Such a criticality measure is non-negative everywhere, zero if and only if
the point is stationary, and typically continuous.
While this is difficult to achieve in our case, we can show that the criticality measure
arising from \cref{prp:stationarity} is lower semi-continuous on bounded subsets of $\BVW(0,1)$
with respect to weak$^*$ convergence in $\BVW(0,1)$ and implies lower bounds on the predicted reduction near
discontinuities to still obtain the claim. We now provide the criticality measure for \eqref{eq:q} and prove that
it is lower semi-continuous.

For $w \in \BVW(0,1)$, $n(w) \in \N$ is the number of switching points and the switching points are
denoted by $t_i$ for $i \in \{1,\ldots,n(w)\}$, where we always assume that they are ordered
as $t_1 < \ldots < t_{n(w)}$. The \emph{criticality measure} of $w$ reads
\begin{gather}\label{eq:criticality}
\crit(w) = \sum_{i=1}^{n(w)} |\nabla F(w)(t_i)(w(t_i^+) - w(t_i^-))|,
\end{gather}
where $w(t_i^+) = \lim_{t\searrow t_i^+} w(t)$ and
$w(t_i^-) = \lim_{t\nearrow t_i^-} w(t)$. Note that the left and right limits 
are well defined because any function in $\BV(0,1)$ can be represented 
by the difference of two monotone functions. $\crit$ in \eqref{eq:criticality}
is equal to the left-hand side in the claim of 
\cref{prp:stationarity}:
\begin{gather}\label{eq:criticality_var}
\crit(w) = \langle |Dw|, |\nabla F(w)|\rangle_{\calM,C}.
\end{gather}
Clearly $\crit : \BVW(0,1) \to [0,\infty)$ is non-negative and zero if and only if the input is stationary.
By means of the second characterization, we show that $\crit$ is weakly$^*$ sequentially lower semi-continuous
under \cref{ass:general_assumption}. In contrast to the verification of the assumptions for the trust-region
analysis below, this proof also works if we assume a multi-dimensional domain $\Omega$ and we thus provide
it in a multi-dimensional setting as a corollary of the following insight.
\begin{lemma}\label{lem:analysis_of_varmeasure_wstar}
Let $\Omega \subset \R^d$ be a bounded Lipschitz domain. Let $w_n \weakstarto w$ in $\BV(\Omega)$. Then
\begin{gather}\label{eq:Cwstarlsc}
\langle |Dw|, \phi \rangle_{\calM,C}
\le \liminf_{n\to\infty}\,
\langle |Dw_n|, \phi\rangle_{\calM,C} 
\end{gather}
holds for $\phi \in C_0(\Omega)$ with $\phi \ge 0$.
\end{lemma}
\begin{proof}
We first note that every weakly$^*$ convergent subsequence $|Dw_{n_k}| \weakstarto \nu$ satisfies
\begin{gather}\label{eq:Cwstarlsc_subseq}
\langle |Dw|, \phi \rangle_{\calM,C}
   \le \langle \nu, \phi\rangle_{\calM,C} = \lim_{k\to \infty} \langle |Dw_{n_k}|, \phi\rangle_{\calM,C}
\end{gather}
if $|Dw| \le \nu$ holds. Because of the inner regularity of Radon measures \cite[Prop.\ 1.43]{ambrosio2000functions}, $|Dw| \le \nu$ holds if and only if $|Dw|(K) \le \nu(K)$ holds for all compact sets $K$.
We first show that $|Dw|(K) \le \nu(K)$ holds for all compact sets $K$ and then show the existence of
suitable subsequences so that \eqref{eq:Cwstarlsc_subseq} implies \eqref{eq:Cwstarlsc}.

\textbf{$|Dw|(A) \le \nu(A)$ holds for all pre-compact sets $A \subset \subset \Omega$:}
Due to the regularity of Radon measures, every measurable set $A \subset \subset \Omega$ can be approximated for all $\varepsilon > 0$
with compact and open sets $K_\varepsilon \subset A \subset U_\varepsilon \subset \subset \Omega$
such that $\mu(U_\varepsilon \setminus K_\varepsilon) \le \varepsilon$ holds for $\mu \in \{|Dw|, \nu\}$.
We define the continuous function $\chi_\varepsilon(x) \coloneqq \frac{\dist(x, \Omega\setminus U_\varepsilon)}{\dist(x, K_\varepsilon) + \dist(x, \Omega\setminus U_\varepsilon)}$ for $x \in \Omega$ so
that $\supp \chi_\varepsilon \subset U_\varepsilon$ and
$\chi_\varepsilon(x) = 1$ for $x \in K_\varepsilon$.

Let $y \in \R^d$ with $\|y\| = 1$ be fixed. Then we can deduce
\begin{align*}
\int_A y \cdot \dd Dw &= 
\int_\Omega \chi_\varepsilon y \cdot \dd Dw + \int_\Omega (\chi_A - \chi_\varepsilon)y \cdot \dd Dw\\
&\leftarrow
\int_\Omega \chi_\varepsilon y \cdot \dd Dw_n + \int_\Omega (\chi_A - \chi_\varepsilon)y \cdot \dd Dw\\
&\le \int_\Omega \chi_\varepsilon \dd |Dw_n| + \int_\Omega \chi_A - \chi_\varepsilon \dd |Dw|
\\
&\to \int_\Omega \chi_\varepsilon \dd \nu + \int_\Omega \chi_A - \chi_\varepsilon \dd |Dw|\\
&= \nu(A)
+ \int_\Omega \chi_A - \chi_\varepsilon \dd |Dw|
+ \int_\Omega \chi_A - \chi_\varepsilon \dd \nu\\
&\le \nu(A) + 2 \varepsilon,
\end{align*}
where the first inequality follows with an approximation
of $\chi_\varepsilon$ as a monotone limit of simple functions.
Since the right-hand side is independent of $y$, we can supremize over
$y \in \R^d$ with $\|y\| = 1$ and obtain
$\|Dw(A)\| \le \nu(A) + 2 \varepsilon$.
Using the fact that $\varepsilon > 0$ was arbitrary, we obtain
\[ \|Dw(A)\| \le \nu(A). \]

It is clear that for every pre-compact set $A \subset \subset \Omega$ we have
\[ |Dw|(A) = \inf\Big\{\sum_{i=1}^m \|Dw(A_i)\| \,:\, A_1,\ldots,A_m \subset \subset \Omega \Big\},  \]
which implies
\[ |Dw|(A) \le |\nu|(A) = \nu(A). \]

\textbf{\eqref{eq:Cwstarlsc_subseq} implies \eqref{eq:Cwstarlsc}:}
For all $\varepsilon > 0$, the Banach--Alaoglu theorem \cite[Thm 1.59]{ambrosio2000functions}
gives the existence of
a weakly$^*$ convergent subsequence $|Dw_{n_k}| \weakstarto \nu$ (sequence and limit depending on $\varepsilon$) such that
\[ \liminf_{n\to\infty} \langle |Dw_n|, \phi\rangle_{\calM,C} \ge \lim_{k\to \infty} \langle |Dw_{n_k}|, \phi\rangle_{\calM,C} - \varepsilon
\]
so that \eqref{eq:Cwstarlsc_subseq} gives
\[ \langle |Dw|, \phi \rangle_{\calM,C} 
   \le \liminf_{n\to\infty}\,\langle |Dw_n|, \phi\rangle_{\calM,C} + \varepsilon
\]
for all $\varepsilon > 0$.
\end{proof}

\begin{theorem}\label{lem:Cwstarlsc}
Let $\Omega \subset \R^d$ be a bounded Lipschitz domain.
Let $X \subset \BVW(\Omega)$ be bounded with respect to $\TV$.
Let $p > d$ if $d \ge 2$ and $p \ge 1$ if $d = 1$. Let \cref{ass:general_assumption} hold.
Let $\crit : \BVW(\Omega) \to [0,\infty)$ be defined through the characterization
in \eqref{eq:criticality}. Then $\crit$ is weakly$^*$ sequentially lower semi-continuous.
\end{theorem}
\begin{proof}
Let $w_n \weakstarto w$ in $\BVW(\Omega)$. Then $w_n \weakstarto w$ in $L^q(\Omega)$ for all $q \ge 1$
because $\BV(\Omega) \hookrightarrow L^{\frac{d}{d-1}}(\Omega)$ \cite[Thm 3.47]{ambrosio2000functions}
and $\{w^n\}_n$ is uniformly in $L^\infty(\Omega)$ since $W$ is finite.
Together with $W^{1,p}(\Omega) \hookrightarrow C(\bar{\Omega})$
\cite[Thm 4.12]{adams2003sobolev}
and \cref{ass:general_assumption} we obtain $\|\nabla F(w_n)\| \to \|\nabla F(w)\|$ in $C(\bar{\Omega})$.
Slightly more involved than the argument above \cref{prp:stationarity}, we can meaningfully define
$\langle |Df|, g\rangle_{\calM, C}$ for $f \in \BVW(\Omega)$ and $g \in C(\bar{\Omega})$ because
$|Df|$ is concentrated on the boundary of the finitely many level sets of $f$ so that 
$\langle |Df|, g\rangle_{\calM, C}$ can be written as a finite sum
\[ \langle |Df|, g\rangle_{\calM, C} = \sum_{i=1}^\frac{M(M - 1)}{2} \delta_i \int_{\Gamma_i} g(t) \dd \Ha^{d-1}(t) \]
for some, where $\Gamma_i$ is a $d-1$-dimensional subset of $\Omega$ and in particular a subset
of the reduced boundary of the level sets of $f$ and $0 \le \delta_i \le \delta_{\max} \coloneqq \max W - \min W$;
see \cite{manns2023on}. Consequently,
\begin{align*}
  \langle |Df|, g\rangle_{\calM, C}
   &\le  \delta_{\max} \sup \big\{ g(t)\,:\, t \in \Omega \big\} \Ha^{d-1}\big(\Gamma_1 \cup \cdots \cup \Gamma_{\nicefrac{1}{2}M(M-1)}\big) \\
   &\le \delta_{\max}\sup \big\{ g(t)\,:\, t \in \bar{\Omega}\big\}\Ha^{d-1}\big(\Gamma_1 \cup \cdots \cup \Gamma_{\nicefrac{1}{2}M(M-1)}\big),
\end{align*}
which gives
\begin{align*}
\langle |Dw_n|, \|\nabla F(w_n)\|\rangle_{\calM,C}
&= \langle |Dw_n|, \|\nabla F(w)\|\rangle_{\calM,C} 
+ \underbrace{\langle |Dw_n|, \|\nabla F(w_n)\| - \|\nabla F(w)\|\rangle_{\calM,C}}_{\to 0},
\end{align*}
where we have applied the estimate above and employed that the sum of the interface lengths between the different
level sets stays bounded for a bounded sequence in $\BVW(\Omega)$; see \cite[Lem.\ 2.1]{manns2023on}.
Thus it remains to show
\[ 
\langle |Dw|,\|\nabla F(w)\| \rangle_{\calM,C}
\le \liminf_{n\to\infty}\,\langle |Dw_n|, \|\nabla F(w)\|\rangle_{\calM,C}
\]

To see the last claim, we multiply $\|\nabla F(w)\|$ with a family of smooth and compactly supported cutoff
functions $\{\psi_k\}_k$ such that $\psi_k \le \psi_{k+1}$ and obtain
\[ \phi_k \to \|\nabla F(w)\| \]
holds pointwise for $\phi_k \coloneqq \psi_k \|\nabla F(w)\|$ and $\phi_k \in C_0(\Omega)$, where
we also have $0 \le \phi_{k} \le \phi_{k+1} \le \|\nabla F(w)\|$ pointwise for all $k$.

For all $k \in \N$ we obtain
\[ \langle |Dw|, \phi_k \rangle_{\calM,C} \le \liminf_{n\to\infty}\,\langle |Dw_n|, \phi_k \rangle_{\calM,C}
\]
from \cref{lem:analysis_of_varmeasure_wstar}. Then Fatou's lemma \cite[Thm 1.20]{ambrosio2000functions}
gives 
\[ \langle |Dw|, \|\nabla F(w)\| \rangle_{\calM,C} 
\le \liminf_{k\to\infty} \liminf_{n\to\infty}\,\langle |Dw_n|, \phi_k\rangle_{\calM,C}.
\]
Using the positivity of $|Dw_n|$ and $0 \le \phi_{k} \le \phi_{k+1} \le \|\nabla F(w)\|$ pointwise,
we obtain
\[ \langle |Dw|, \|\nabla F(w)\| \rangle_{\calM,C} \le \liminf_{n\to\infty}\,\langle |Dw_n|, \|\nabla F(w)\|\rangle_{\calM,C}.
\]
\end{proof}

\section{Abstract trust-region algorithm analysis}\label{sec:abstract}
In this section, we provide a variant of \cref{alg:trm} for optimizing
\eqref{eq:p_abstract} as \cref{alg:trm_abstract}. We impose assumptions on
$\crit : X \to [0,\infty)$ as well as $\pred : X \times [0,\infty) \to \R$ and
$\ared : X \times X \to \R$ that occur in \cref{alg:trm_abstract} and
show $\crit(x_n) \to 0$ in \cref{thm:Ctozero}. This then implies that
all limit points of \cref{alg:trm_abstract} are stationary
if $\crit$ is a lower-semicontinuous criticality measure, that is, $\crit(x) = 0$ if and only if $x$ is 
stationary (and feasible) for \eqref{eq:p_abstract}; see \cref{cor:stationary}.
\begin{algorithm}[t]
	\caption{Abstract trust-region algorithm to optimize \eqref{eq:p_abstract}
	without trust-region radius reset}\label{alg:trm_abstract}
	\textbf{Input:} $x_0 \in X$, $0 < \sigma < 1$, $\Delta_{\max} \in (0,\infty]$.
	
	\begin{algorithmic}[1]
		\For{$n = 0,1,2\ldots$}
		\State $\tilde{x}_n \gets $ Solve trust-region subproblem.\label{ln:tr}
		\State $\pred_{n} \gets \pred(x_n, \Delta_n)$
		\State $\ared_{n} \gets \ared(x_n,\tilde{x}_n) = J(x_n) - J(\tilde{x}_n)$
		\If{$\pred_n = 0$}
		\State Terminate and return $x_n$.
		\ElsIf{$\ared_{n} \ge \sigma \pred_{n}$}\label{ln:sufficient_decrease}
		\State $x_{n+1} \gets \tilde{x}_n$
		\State $\Delta_{n+1} \gets \min\{2\Delta_n,\Delta_{\max}\}$
		\Else
		\State $x_{n+1} \gets x_n$
		\State $\Delta_{n+1} \gets 0.5 \Delta_n$
		\EndIf
		\EndFor
	\end{algorithmic}
\end{algorithm}
\begin{assumption}\label{ass:jump_nonsmoothness}
Let $(X, d)$ be a metric space. We assume that $\pred$ is monotonically increasing in the second
argument if the first argument is fixed. In addition, we assume the following properties of $\crit$,
$\pred$, and $\ared$.
\begin{enumerate}
	\item There exist $c_0 > 0$, $c_1 \ge 0$, $s \in (0,1)$, and $\underline{\Delta}_a : X \to [0,\infty)$ such that
	\[ \pred(x,\Delta) \ge c_0 \crit(x)\Delta - c_1 \Delta^{1 + s} \]
	for all $\Delta \le \underline{\Delta}_a(x)$ for all $x \in X$.
	\item In addition to $c_0$, $c_1$, $s$, $\underline{\Delta}_a : X \to [0,\infty)$ from above, there exist
	$c_2 \ge c_1$, $\underline{\Delta}_b > 0$, and $\delta > 0$ such that for
	all iterates $x_n$, $x_{n+1}$ produced by \cref{alg:trm_abstract} and
	\[ R_n \coloneqq (1 - \sigma)\pred(x_n,\Delta_n)  - |\ared(x_n,x_{n+1}) - \pred(x_n,\Delta_n)|
	\]
	we have
	\begin{align*}
	R_n &\ge (1 - \sigma)\Big(c_0 \crit(x_n) \Delta_n - c_2 \Delta_n^{1 + s}\Big)
	&&\text{if } \Delta_n \le \underline{\Delta}_a(x_n),\\
	\pred(x_n,\Delta_n) &\ge \delta 
	&&\text{if } \underline{\Delta}_a(x_n) \le \min\{\Delta_n,\underline{\Delta}_b\},\\
	R_n &\ge \delta
	&&\text{if } \underline{\Delta}_a(x_n) \le \Delta_n \le \underline{\Delta}_b.
	\end{align*}
	\item In addition to the constants defined above, there exist $\underline{\Delta}_c > 0$ and $L > 0$
	such that $d(x_n,x_{n+1}) \le \Delta_n \le \underline{\Delta}_c$ for two subsequent iterates produced
	by \cref{alg:trm_abstract} implies
	\[ \pred(x_n,\Delta_n) \ge \delta \quad\text{or}\quad |\crit(x_n) - \crit(x_{n+1})| \le L d(x_n, x_{n+1}). \]	
\end{enumerate}
\end{assumption}
\begin{remark}\label{rem:additional_terms}
If the $\Delta_n$-dependent
lower bounds on $\pred(x,\Delta)$
and $R_n$ in \cref{ass:jump_nonsmoothness} 1.,2.\ above can be shown with additional negative terms of higher order or if we have $s > 1$
in these estimates, then we can still verify \cref{ass:jump_nonsmoothness} 1.,2.
To this end, we just have
to increase
$c_1$, $c_2$ and reduce
$\underline{\Delta}_a$, $\underline{\Delta}_b$
until we meet the criteria
again (for some $s \in (0,1)$)
because
$\tfrac{\Delta^{q}}{\Delta^{p}}
\to 0$ as $\Delta \to 0$ if $q > p > 0$.
\end{remark}

\Cref{ass:jump_nonsmoothness} 1.\ and 2.\ are a partial substitute for a Cauchy or sufficient decrease condition as is typical in the
convergence analysis of trust-region algorithms. \Cref{ass:jump_nonsmoothness} 1.\ means that the predicted reduction is bounded below
by a scalar multiple of the criticality measure times the trust-region radius provided the trust-region radius is
small enough, implying that the predicted reduction behaves at least proportional to the trust-region radius, thereby enforcing
large enough steps. The term $c_1 \Delta^{1+s}$ can be used to encapsulate higher-order terms as may arise from remainder estimates of
Taylor's theorem. In our case, \emph{small enough} depends on the current iterate and this upper bound is used to model the feasibility limits
of the decrease steps that can be taken based on the value of the criticality measure. For \eqref{eq:q}, this means that a switch can
only be shifted to the left or right until the boundary of the domain is reached or a switch with opposite sign occurs without losing
control over the behavior of the $\TV$-term. In Hilbert space trust-region methods for smooth problems over convex and closed
feasible sets, \cref{ass:jump_nonsmoothness} 1.\ and 2.\ are implied by Taylor's theorem and Cauchy decrease conditions; see (40) and (60)
and the comments in \cite{toint1988global}.

\Cref{ass:jump_nonsmoothness} 2.\ provides a lower bound on the 
the remainder term $R_n$ that needs to be positive for an iteration
to be successful. 
Importantly, it also implies positive lower bounds on the predicted
reduction and $R_n$ if the trust-region radius is small but large enough
with respect to the current iterate. In our algorithm analysis, this will provide a
means to handle the situation that the trust-region radius decreases too fast for
\cref{ass:jump_nonsmoothness} 1.\ to guarantee sufficient decrease.
For \eqref{eq:q}, this will be verified by exploiting that
the removal of a switch decreases the objective significantly if
the trust-region radius is small.

\Cref{ass:jump_nonsmoothness}  3.\ is a conditional continuity assumption that means that if the
criticality measure changes significantly between two close enough iterates, the predicted reduction
is at least a fixed constant. In other words, discontinuities in the criticality measure that might otherwise
lead to $\crit(x_n) \to 0$ too fast are compensated by the behavior of the predicted reduction, which is bounded below
by a fixed constant in this case and thus induces a fixed improvement of the objective value.
In \eqref{eq:q}, this situation can happen when
a new switch occurs from one iterate to the next and thus a new positive term
appears in $\crit(x_{n+1})$ compared to $\crit(x_n)$, see the characterization
in \eqref{eq:criticality}. Since $\crit$ is generally continuous in a Hilbert space trust-region methods for smooth problems over
convex and closed feasible sets, \cref{ass:jump_nonsmoothness} is not necessary in such a setting. This is due to the continuity
of the projection onto convex and convex sets in Hilbert spaces; see (4) and (12) and the comments
in \cite{toint1988global}.

Before starting our analysis of the asymptotics, we provide an auxiliary lemma.
\begin{lemma}\label{lem:Ctozero_prep}
Let \cref{ass:jump_nonsmoothness} hold. Let $\{n_k\}_k$ be a subsequence of successful
iterations of \cref{alg:trm_abstract}, that is, $\{n_k\}_k \subset \{n \in \N: \ared_n \ge \sigma \pred_n \}$.
If $\limsup_{k\to\infty} \Delta_{n_k} > 0$ and $\liminf_{k\to\infty} \crit(x_{n_k}) > 0$ hold,
then there exists $\delta > 0$ such that
\[ \pred(x_{n_{k_\ell}}, \Delta_{n_{k_\ell}}) \ge \delta \]
holds for an infinite subsequence $\{n_{k_\ell}\}_\ell \subset \{n_k\}_k$.
\end{lemma}
\begin{proof}
Let $\varepsilon \coloneqq \liminf_{k\to\infty} \crit(x_{n_k})$. Then \cref{ass:jump_nonsmoothness} 1.\ and the
montonicity of $\pred(x_{n_k}, \cdot)$ give
\[
\pred(x_{n_k},\Delta_{n_k}) 
\ge c_0 \varepsilon \min\{\Delta_{n_k}, \underline{\Delta}_a(x_{n_k})\} 
  - c_1 \min\{\Delta_{n_k}, \underline{\Delta}_a(x_{n_k})\}^{1 + s}.
\]
After choosing a suitable infinite subsequence $\{n_{k_\ell}\}_\ell \subset \{n_k\}_k$, we can assume that
there is $\underline{\Delta} > 0$ such that $\Delta_{n_{k_\ell}} \ge \underline{\Delta}$ holds for all
$\ell \in \N$. 

We make a case distinction and start with the case $\liminf_{k\to\infty} \underline{\Delta}_a(x_{n_{k_\ell}}) > 0$.
Then, we reduce $\underline{\Delta}$ and pass to a subsequence (for ease of notation denoted by the same symbol)
such that $\underline{\Delta}_a(x_{n_{k_\ell}}) \ge \underline{\Delta}$ also holds for all $\ell \in \N$.
Consequently, the monotonicity of $\pred(x_{n_{k_\ell}}, \cdot)$ also gives
\[ \pred(x_{n_{k_\ell}}, \Delta_{n_{k_\ell}}) \ge \max\{ c_0 h - c_1 h^{1 + s}\,:\, 0 \le h \le \underline{\Delta} \} \eqqcolon \delta, \]
which is strictly positive because $c_1 h^{1+s}$ is in $o(h)$.

Second, we consider the case $\liminf_{k\to\infty} \underline{\Delta}_a(x_{n_{k_\ell}}) = 0$. In this case, we
pass to a subsequence (for ease of notation denoted by the same symbol) such that
$\Delta_{a}(x_{n_{k_\ell}}) \le \min\{\Delta_{n_{k_\ell}},\underline{\Delta}_b\}$ holds for all $\ell \in \N$.
Then \cref{ass:jump_nonsmoothness} 2.\ gives
\[ \pred(x_{n_{k_\ell}}, \Delta_{n_{k_\ell}}) \ge \delta
\]
for some $\delta > 0$.	
\end{proof}

Although our proof of $\crit(x_n) \to 0$ will not require this intermediate result explicitly, we provide a
short proof that $\liminf_{n\to\infty} \crit(x_n) = 0$ holds over the course of the iterations
for the sake of completeness.
\begin{lemma}\label{lem:liminf}
Let $J$ be bounded below.
Let \cref{ass:jump_nonsmoothness} hold.
Let $\{x_n\}_n$ denote the sequence of iterates produced by \cref{alg:trm}.
Let $\{x_n\}_n$ be infinite. Then $\liminf_{n\to\infty} \crit(x_n) = 0$.
\end{lemma}
\begin{proof}
We assume by way of contradiction that there are $\varepsilon > 0$ and $n_0 \in \N$ such that
\[ \crit(x_n) > \varepsilon \]
holds for all $n \ge n_0$. Let the successful iterations be denoted by $\{n_k\}_k$.
If $\limsup_{k\to\infty} \Delta_{n_k} > 0$, then \cref{lem:Ctozero_prep} implies 
\[ \pred(x_{n_{k_\ell}},\Delta_{n_{k_\ell}}) \ge \delta \]
for an infinite subsequence indexed by $\ell$, which gives the contradiction
\[ J(x_0) - \lim_{n\to\infty} J(x_n) \ge
\sum_{\ell=1}^\infty \ared(x_{n_{k_\ell}},x_{{n_{k_\ell}}+1})
\ge \sigma \sum_{\ell=1}^\infty \pred(x_{n_{k_\ell}},\Delta_{{n_{k_\ell}}+1})
\ge \sigma \sum_{\ell=1}^\infty \delta = \infty
\]
because we have assumed that $J$ is bounded below.

Consequently, we have  $\lim_{k\to\infty} \Delta_{n_k} = 0$ for the successful iterations
and in turn $\lim_{n\to\infty} \Delta_{n} = 0$ because the trust-region radius only increases
in successful iterations.

If there is an infinite subsequence $\{n_\ell\}_\ell$ of iterations such that
\[ \underline{\Delta}_a(x_{n_\ell}) \le \Delta_{n_\ell} \to 0 \]
holds, then eventually $\underline{\Delta}_a(x_{n_\ell}) \le \Delta_{n_\ell} \le \underline{\Delta}_b$ 
holds so that for all small enough $\Delta_{n_\ell}$ and thus for all large enough $\ell$
we have $R_{n_\ell} \ge \delta$ for $R_{n_\ell}$ from \cref{ass:jump_nonsmoothness} 2.\
and thus a successful iteration. Consequently, we obtain from \cref{ass:jump_nonsmoothness} 2.\
that
\[ \ared(x_{n_{\ell}},x_{n_\ell +1}) \ge \sigma \pred(x_{n_{\ell}},\Delta_{n_{\ell}}) 
\ge \sigma\delta \]
holds for infinitely many $\ell \in \N$. As above, this contradicts that $J$ is bounded below.

Consequently, we can wlog assume $\Delta_{n} \le \underline{\Delta}_a(x_{n})$
and $\Delta_{n} \le \underline{\Delta}_b$ for all large enough $n$.
Because $\crit(x_n) > \varepsilon$ holds, \cref{ass:jump_nonsmoothness} 2.\ gives
$R_n \ge 0$ for the $R_n$ from \cref{ass:jump_nonsmoothness} 2.\ 
whenever the trust-region satisfies radius satisfies
$\Delta_{n} \le \underline{\Delta}$ for some fixed and small enough $\underline{\Delta} > 0$.
This implies that all iterations $n$ are successful if $n$ is large enough.
Then the trust-region update rule in \cref{alg:trm} gives
$\Delta_n \to \infty$, which contradicts $\Delta_n \to 0$ and closes the proof.
\end{proof}

\begin{theorem}\label{thm:Ctozero}
Let $J$ be bounded below.
Let \cref{ass:jump_nonsmoothness} hold.
Let $\{x_n\}_n$ denote the sequence of iterates produced by \cref{alg:trm}.
Let $\{x_n\}_n$ be infinite. Then $\lim_{n \to \infty} \crit(x_n) = 0$.
\end{theorem}
\begin{proof}
We begin by following the proof strategy of Theorem 6 in \cite{toint1997non} and Theorem 4.4 in 
\cite{manns2023convergence}. Several modifications are necessary, however, since we have to
substitute the continuity of the criticality measure by the properties asserted
in \cref{ass:jump_nonsmoothness}.
Let $S = \{ n \in \N\,:\, \ared(x_n,x_{n+1}) \ge \sigma \pred(x_n,\Delta_n) \}$, that is,
$S$ is the set of successful iterations. For all $n \in S$ it holds that
\begin{gather}\label{eq:ared_from_ass_1}
\ared(x_n,x_{n+1})
\ge \sigma \pred(x_n, \Delta_n) 
 \ge \sigma\big(c_0 \crit(x_n) \min\{\Delta_n, \underline{\Delta}_a(x_n)\} - c_1 \min\{\Delta_n, \underline{\Delta}_a(x_n)\}^{1 + s}\big)
\end{gather}
by virtue of \cref{ass:jump_nonsmoothness} 1.

We assume by way of contradiction that there are $\varepsilon > 0$ and an 
infinite subsequence of successful iterations $\{n_k\}_{k} \subset S$
such that
\[ \crit(x_{n_k}) > \varepsilon > 0 \]
holds for all $k \in \N$. Our goal is to exclude all situations by showing that each of them would
lead to $J(x_n) \to -\infty$. Note that it is sufficient to consider successful iterations here since
$x_n$ can only change in a successful iteration and there are infinitely many of them by assumption.

If $\limsup_{k\to\infty} \Delta_{n_k} > 0$, then \cref{lem:Ctozero_prep} gives
an infinite subsequence $\{n_{k_\ell}\}_\ell \subset \{n_k\}_k$ such that
\[ 
J(x_0) - \lim_{n\to\infty} J(x_n)
\ge \sum_{k=1}^{\infty} \ared(x_{n_k}, x_{n_k + 1})
\ge \sigma \sum_{k=1}^{\infty} \pred(x_{n_k}, \Delta_{n_k})
\ge \sigma \sum_{\ell=1}^{\infty} \delta = \infty
\]
holds, which contradicts that $J$ is bounded below. Consequently, $\Delta_{n_k} \to 0$ must hold.
We make a case distinction on the relationship between $\Delta_{n_k}$ and
$\underline{\Delta}_a(x_{n_k})$.

\textbf{Case $\underline{\Delta}_a(x_{n_k}) \le \Delta_{n_k}$ for all $k \ge k_0$ and some $k_0 \in \N$:}\quad
In this case, we obtain from $\Delta_{n_k} \to 0$ that there exists $k_1 \ge k_0$ such that
$\underline{\Delta}_a(x_{n_k}) \le \Delta_{n_k} \le \underline{\Delta}_b$ holds
for $\underline{\Delta}_b > 0$ from \cref{ass:jump_nonsmoothness} 2.\ and all $k \ge k_1$.
Consequently, we obtain from \cref{ass:jump_nonsmoothness} 2.\
\[ 
J(x_0) - \lim_{n\to\infty} J(x_n)
\ge \sigma \sum_{k=1}^{\infty} \pred(x_{n_k}, \Delta_{n_k})
\ge \sigma \sum_{k=k_0}^{\infty} \delta = \infty,
\]
which contradicts that $J$ is bounded below.

\textbf{Case $\Delta_{n_{k_\ell}}\le \underline{\Delta}_a(x_{n_{k_\ell}})$ for an infinite subsequence
$\{n_{k_\ell}\}_\ell \subset \{n_k\}_k$:}\quad Since we only need to work with this subsequence now,
we denote it by the symbol $n_\ell$ instead of $n_{k_\ell}$
from now on to avoid notational bloat. Moreover, it is sufficient to consider
the situation $\Delta_{n_\ell} \le \underline{\Delta}_b$ because $\Delta_{n_\ell} \to 0$. Then we obtain from \cref{ass:jump_nonsmoothness} 1.\ and 2.\ 
\begin{align*}
\pred(x_{n_\ell}, \Delta_{n_\ell})
\ge \sigma \sum_{\ell = 1}^{\infty} c_0 \crit(x_{n_\ell}) \Delta_{n_\ell} - c_1\Delta_{n_\ell}^{1 + s}.
\end{align*}
This estimate alone is not enough to obtain the contradiction since $\Delta_{n_\ell}$ might tend to zero
to fast. We therefore set forth to estimate the sum of the actual reductions achieved in iteration $n_\ell$
and subsequent successful iterations in order to again obtain a finite lower bound of achieved
actual reductions which happen infinitely often.

To this end, we again use $\Delta_{n_\ell} \to 0$ to pass to a further
subsequence such that the conditions
\begin{align}
\Delta_{n_\ell}\big(2 L + a(n_\ell)\big) &< \varepsilon < \crit(x_{n_\ell})
\label{eq:Deltankprime_epsilon_Cnkprime} \\
0 &< \delta_1 \coloneqq
\frac{1}{3}c_0 a(n_\ell) \Big(\frac{\varepsilon}{2 L + a(n_\ell)}\Big)^2 
- \frac{1}{2^{1 + s}} c_1 \Big(\frac{\varepsilon}{2 L + a(n_\ell)}\Big)^{1 + s}
\label{eq:qredsum_positivity_below_1} \\
0 &< \delta_2 \coloneqq
\frac{1}{3}c_0 a(n_\ell) \min\{\Delta_{\max},\underline{\Delta}_b,\underline{\Delta}_c\}^2 
- \frac{1}{2^{1 + s}} c_1 \min\{\Delta_{\max},\underline{\Delta}_b,\underline{\Delta}_c\}^{1 + s}
\label{eq:qredsum_positivity_below_2} \\
\Delta_{n_\ell} &\le \min\{\underline{\Delta}_b, \underline{\Delta}_c\}
\end{align}
hold for all $\ell \in \N$, where $\underline{\Delta}_b$, $c_0$, $c_1 > 0$, and $s \in (0,1)$ are from 
\cref{ass:jump_nonsmoothness} 1.,2.\ and
\[ a(n_\ell) \coloneqq
\frac{c_2}{c_0}\Delta_{n_\ell}^{s - 1}
\]
Note that the positivity in \eqref{eq:qredsum_positivity_below_1} and \eqref{eq:qredsum_positivity_below_2}
can be ensured because $a(n_\ell) \to \infty$ for $\Delta_{n_\ell} \to 0$. In \eqref{eq:qredsum_positivity_below_1},
the first summand is asymptotically in $\Theta\big( a(n_\ell)^{-1}\big)$ and the second in
$\Theta\big(a(n_\ell)^{- (1 + s)}\big)$ so that the second term tends to zero faster.

We now consider the successful iterations from $n_\ell$ on until the first unsuccessful iteration
$r_\ell = \min\{ n > n_\ell\,:\, \ared(x_n,x_{n+1}) < \sigma \pred(x_n,\Delta_n)\}$.
The iteration $r_\ell$ is well defined (a finite $\min$ exists) because otherwise the trust-region radius
update in successful iterations contradicts $\Delta_{n_\ell} \to 0$. 

We will show that there always is a smallest iteration $\underline{n}_\ell \in \{n_\ell,\ldots,r_\ell - 1\}$
such that
\begin{gather}\label{eq:n_nkprime_first_log_condition}
\underline{n}_\ell - n_\ell > \log_2\bigg(\frac{\min\{\Delta_{\max},\underline{\Delta}_b,\underline{\Delta}_c\}}{\Delta_{n_\ell}}\bigg)
\end{gather}
or
\begin{gather}\label{eq:n_nkprime_second_log_condition}
\underline{n}_\ell - n_\ell > \log_2\frac{\crit(x_{n_\ell})}{\Delta_{n_\ell}\big(2 L + a(n_\ell)\big)} 
\underset{\eqref{eq:Deltankprime_epsilon_Cnkprime}}\ge 1
\end{gather}
or
\begin{gather}\label{eq:n_nkprime_third_log_condition}
\pred(\underline{n}_\ell,\Delta_{\underline{n}_\ell}) \ge \delta
\end{gather}
holds (whichever happens first). To verify this claim, we assume
that \eqref{eq:n_nkprime_first_log_condition}, \eqref{eq:n_nkprime_second_log_condition}, and
\eqref{eq:n_nkprime_third_log_condition} do not hold in iteration
$\tilde{n} \ge n_\ell$ and then deduce that iteration $\tilde{n}$ is successful, which
implies $\underline{n}_\ell < r_\ell$, the first unsuccessful iteration after $n_\ell$.
As a consequence, $\underline{n}_\ell$ is well defined.

For said iterations $\tilde{n} \in \{n_\ell,\ldots,\underline{n}_\ell\}$ we deduce inductively
\[ d(x_{\tilde{n}},x_{\tilde{n}+1}) \le \Delta_{\tilde{n}} = \Delta_{n_\ell}2^{\tilde{n} - n_\ell}
\le \min\{\underline{\Delta}_b, \underline{\Delta}_c\} \]
from the violation of \eqref{eq:n_nkprime_first_log_condition} and
\begin{gather}\label{eq:Ctilden_estimate}
\begin{aligned}
\crit(x_{\tilde{n}})
&\ge \crit(x_{n_\ell}) - \sum_{i=1}^{\tilde{n} - n_\ell} L d(x_{n_\ell + i}, x_{n_\ell + i + 1}) \\
&\ge \crit(x_{n_\ell}) - L \Delta_{n_\ell} \sum_{i=1}^{\tilde{n} - n_\ell} 2^i \\
&\ge \crit(x_{n_\ell}) - L \Delta_{n_\ell} 2^{\tilde{n} - n_\ell + 1} \\
&= \crit(x_{n_\ell}) - 2 L \Delta_{\tilde{n}}
\end{aligned}
\end{gather}
by virtue of \cref{ass:jump_nonsmoothness} 3.

The violation of \eqref{eq:n_nkprime_second_log_condition} gives
\[ \log_2\frac{\crit(x_{n_\ell})}{\Delta_{n_\ell}\big(2 L + a(n_\ell)\big)} \ge \tilde{n} - n_\ell,
\]
which is equivalent to
\[ \crit(x_{n_\ell}) \ge \Delta_{\tilde{n}}(2 L + a(n_\ell)).
\]
Inserting this into \eqref{eq:Ctilden_estimate}, we obtain
\begin{gather}\label{eq:Cnk_minus_radius_estimate_kappa_radius}
\crit(x_{\tilde{n}}) \ge \crit(x_{n_\ell}) - 2 L \Delta_{\tilde{n}}
= \crit(x_{n_\ell}) - L \Delta_{n_\ell} 2^{\tilde{n} - n_\ell + 1}
\ge a(n_\ell)\Delta_{\tilde{n}}.
\end{gather}
Because \eqref{eq:n_nkprime_third_log_condition} and \eqref{eq:n_nkprime_first_log_condition}
are violated, \cref{ass:jump_nonsmoothness} 2.\ implies
$\Delta_{\tilde{n}} \le \underline{\Delta}_a(x_{\tilde{n}})$.
As a consequence, \cref{ass:jump_nonsmoothness} 2.\ also gives for $R_{\tilde{n}}$ defined therein
\[ R_{\tilde{n}} \ge (1- \sigma)
\big(c_0 a(n_\ell) \Delta_{\tilde{n}}^2 - c_2 \Delta_{\tilde{n}}^{1 + s}\big).
\]
By definition of $R_{\tilde{n}}$, iteration $\tilde{n}$ is successful if we can show
\begin{gather}\label{eq:toshow_for_Rn_nonnegative}
c_0 a(n_\ell) \Delta_{\tilde{n}}^2 - c_2 \Delta_{\tilde{n}}^{1 + s}\ge 0.
\end{gather}
Inspecting the definition of $a(n_\ell)$ gives
that \eqref{eq:toshow_for_Rn_nonnegative} follows if
\[ \Delta_{\tilde{n}} \ge \left(\frac{c_2}{c_0 a(n_\ell)}\right)^{\frac{1}{1 - s}} 
 = \Delta_{n_\ell},
\]
which holds true inductively. Consequently, $\tilde{n}$ is successful and in turn
$\underline{n}_\ell < r_\ell$ is well defined.

To work towards a contradiction, we recall that because $\Delta_{\tilde{n}} \le \underline{\Delta}_a(x_{\tilde{n}})$,
\begin{gather}\label{eq:pred_inequality}
\pred(x_{\tilde{n}},\Delta_{\tilde{n}})
   \ge c_0 \crit(x_{\tilde{n}})\Delta_{\tilde{n}} - c_1 \Delta_{\tilde{n}}^{1 + s}
\end{gather}
holds by virtue of \cref{ass:jump_nonsmoothness} 1. We deduce
\begin{align*}
J(x_{n_\ell}) - J(x_{\underline{n}_\ell})
&\ge \sum_{i=0}^{\underline{n}_\ell - n_\ell - 1} \ared(x_{n_\ell +i },x_{n_\ell + i + 1}) \\
&\ge \sigma \sum_{i=0}^{\underline{n}_\ell - n_\ell - 1} \pred(x_{n_\ell + i},\Delta_{n_\ell + i}) \\
&\underset{\mathclap{\eqref{eq:pred_inequality}}}\ge
\sigma \sum_{i=0}^{\underline{n}_\ell - n_\ell - 1} c_0 C(x_{n_\ell + i})\Delta_{n_\ell + i} - c_1\Delta_{n_\ell + i}^{1+s} \\
&\underset{\mathclap{\eqref{eq:Cnk_minus_radius_estimate_kappa_radius}}}\ge
\sigma\sum_{i=0}^{\underline{n}_\ell - n_\ell - 1} c_0 a(n_\ell) \Delta_{n_\ell + i}^2 - c_1\Delta_{n_\ell + i}^{1+s}\\
&= \sigma \bigg(c_0 a(n_\ell) \Delta_{n_\ell}^2 \sum_{i=0}^{\underline{n}_\ell-n_\ell-1} 4^i 
               - c_1 \Delta_{n_\ell}^{1 + s}\sum_{i=0}^{\underline{n}_\ell-n_\ell-1} \big(2^{1 + s}\big)^i\bigg) 
&&\parbox{7em}{\scriptsize $n_{\ell}+ i$ successful for\\  $i = 0,\ldots,\underline{n}_\ell - n_\ell - 1$\\ $\Rightarrow \Delta_{n_\ell + i} = \Delta_{n_\ell}2^i$} \\
&= \sigma \bigg(\underbrace{\frac{1}{3}c_0 a(n_\ell) (\Delta_{n_\ell} 2^{\underline{n}_\ell-n_\ell})^2 
- \frac{1}{2^{1 + s}-1} c_1 (\Delta_{n_\ell}2^{\underline{n}_\ell-n_\ell} )^{1 + s}}_{\eqqcolon Q(\Delta_{n_\ell}2^{\underline{n}_\ell-n_\ell})}\bigg)
\end{align*}
Because $Q(\Delta_{n_\ell}2^{\underline{n}_\ell-n_\ell})$ is a sum of strict upper bounds on $R_{\tilde{n}}$ for
$\tilde{n} \in \{n_\ell,\ldots,\underline{n}_\ell-1\}$ and the $R_{\tilde{n}}$ are all positive as argued above,
we obtain that $Q(\Delta_{n_\ell}2^{\underline{n}_\ell-n_\ell})$ is positive for $\underline{n}_\ell \ge n_\ell$.
Moreover, $Q$ is monotonically increasing on $Q^{-1}((0,\infty))$ because its minimizer/the zero of
its derivative has a negative value of $Q$ since the derivative of the second term tends to $-\infty$ when
the input of $Q$ tends to zero and $Q(0) = 0$.
Consequently, $Q$ is monotonically increasing if its input is
larger than $\Delta_{n_\ell}$.

Now, we consider the three possible cases for $\underline{n}_\ell$ and show that in every one of them
$J(x_{n_\ell}) - J(x_{\underline{n}_\ell})$ is bounded below a strictly positive constant.
Consequently, none of them can occur infinitely often for the sequence $\{n_\ell\}_\ell$.
Because $\underline{n}_\ell$ is well defined, that is finite, for all $\ell \in \N$ as argued above,
at least one of them has to occur infinitely often, which is not possible and thus gives the
final contradiction concluding the proof.

We start with \eqref{eq:n_nkprime_first_log_condition}. Then we can estimate
\begin{align*}
J(x_{n_\ell}) - J(x_{\underline{n}_\ell})
&\ge \sigma \bigg(\frac{1}{3}c_0 a(n_\ell) \min\{\Delta_{\max},\underline{\Delta}_b,\underline{\Delta}_c\}^2 
- \frac{1}{2^{1 + s}} c_1 \min\{\Delta_{\max},\underline{\Delta}_b,\underline{\Delta}_c\}^{1 + s}\bigg) = \sigma \delta_2 > 0
\end{align*}
by means of \eqref{eq:qredsum_positivity_below_2}.

We continue with \eqref{eq:n_nkprime_second_log_condition}. Then we can estimate
\begin{align*}
J(x_{n_\ell}) - J(x_{\underline{n}_\ell})
&\ge \sigma \bigg(\frac{1}{3}c_0 a(n_\ell) \Big(\frac{\crit(x_{n_\ell})}{2 L + a(n_\ell)}\Big)^2 
- \frac{1}{2^{1+s} - 1} c_1 \Big(\frac{\crit(x_{n_\ell})}{2 L + a(n_\ell)}\Big)^{1 + s}\bigg) \\
&\ge \sigma \bigg(\frac{1}{3}c_0 a(n_\ell) \Big(\frac{\varepsilon}{2 L + a(n_\ell)}\Big)^2 
- \frac{1}{2^{1 + s} - 1} c_1 \Big(\frac{\varepsilon}{2 L + a(n_\ell)}\Big)^{1 + s}\bigg) 
= \sigma \delta_1 > 0
\end{align*}
by means of \eqref{eq:qredsum_positivity_below_1}, where the first inequality follows from the
monotonicity of $Q$ and the second inequality with a similar monotonicity argument for
$t \mapsto a t^{2} - b t^{1 + s}$
for $a$, $b > 0$
and a similar monotonicity argument as above for $Q$.

Finally, \eqref{eq:n_nkprime_third_log_condition} implies
\[ J(x_{n_\ell}) - J(x_{\underline{n}_\ell}) \ge \delta. \]
\end{proof}

\begin{corollary}\label{cor:stationary}
Let $\crit : X \to [0,\infty)$ be lower semi-continuous. Let \cref{ass:jump_nonsmoothness} be satisfied.
If \cref{alg:trm_abstract} produces a finite sequence of iterates, the last iterate $\bar{x}$ satisfies
$\crit(\bar{x}) = 0$. If \cref{alg:trm_abstract} produces an infinite sequence of iterates, every accumulation
point $\bar{x}$ satisfies $\crit(\bar{x}) = 0$. If $(X,d)$ is a compact metric space, there is at least one
accumulation point.
\end{corollary}
\begin{proof}
\Cref{ass:jump_nonsmoothness} 1.\ directly implies that $\pred(x_n,\Delta_n)$ is strictly positive if
$\crit(x_n) > 0$ and \cref{alg:trm_abstract} can thus not terminate if $\crit(x_n) > 0$. The
lower-semicontinuity of $\crit$ and \cref{thm:Ctozero} prove that every accumulation
point is stationary. If there are infinitely many iterations, the compactness of $X$ implies that there
is at least one accumulation point. 
\end{proof}

\section{Verification of \cref{ass:jump_nonsmoothness} for \eqref{eq:q}}\label{sec:verification}
In this section, we provide the arguments that verify \cref{ass:jump_nonsmoothness} 
under \cref{ass:general_assumption} for the setting that the domain $\Omega$ in 
\eqref{eq:q} is one-dimensional, specifically $\Omega = (0,1)$. This verification
is the claim of \cref{thm:verification_ass_jump_nonsmoothness} and in particular we
show several lemmas that make up its proof.

We can first prove that the number of switches of the iterates $w_n$
produced by \cref{alg:trm} stays bounded. With this property at hand,
we can then proceed to verify \cref{ass:jump_nonsmoothness}.
\begin{lemma}\label{lem:number_of_switches_for_large_n}
Let \cref{ass:general_assumption} hold.
Let $\{w_n\}_n$ be the sequence of iterates produces by \cref{alg:trm}.
Then there exists $n_{\max} \in \N$ such that
\[ n(w_n) \le n_{\max} \]
holds for all $n \in \N$. Moreover,
\[ \TV(w_n) \le (\max W - \min W)n_{\max}.
\]
\end{lemma}
\begin{proof}
Since $\TV(w_n)$ is the sum of the jump heights of $w_n$ for
one-dimensional domains and the minimum jump height
is one, we have $n_{\max} \le \sup_{n \in \N} \TV(w_n)$.

As a consequence, we obtain
\[ n_{\max} \le \sup_{n \in \N} \TV(w_n) \le 
\sup_{n \in \N} F(w_0) + \TV(w_0) - F(w_n)
\le F(w_0) + \TV(w_0) - \inf_{\mathclap{w \in \BVW(0,1)}} F(w),
\]
where the second inequality holds because
\cref{alg:trm} produces a sequence of iterates with monotonically
non-increasing objective values.
The right hand side is finite because $F$ is bounded below by virtue of
by \cref{ass:general_assumption}.

The second claim follows again from the characterization of
$\TV(w_n)$ as the sum of the jump heights of $w_n$.
\end{proof}
Our main result is that \cref{ass:jump_nonsmoothness} holds for \eqref{eq:q}.
Because of the second claim of \cref{lem:number_of_switches_for_large_n}
and the lower semi-continuity of $\TV$, we can wlog replace the feasible
set by the metric space $(X,d)$ defined by
\[
\begin{aligned}
X &\coloneqq \{ w\in \BVW(0,1)\,:\, n(w) \le n_{\max}\},\\
d(u,v) &\coloneqq \|u - v\|_{L^1(0,1)}\quad\text{ for } u,v \in X.
\end{aligned}
\]
\begin{theorem}\label{thm:verification_ass_jump_nonsmoothness}
Let \cref{ass:general_assumption} hold with $X$ as above and $p = \infty$.
Let $\crit$ be as in \eqref{eq:criticality}.
Then \cref{ass:jump_nonsmoothness} holds.
\end{theorem}
\begin{proof}
It is clear that $(X,d)$ is a metric space and $\crit(w) = 0$ if and only if $w$ is stationary for \eqref{eq:q} in the sense of \cite{leyffer2022sequential}, see \S4.2 therein. We prove
\cref{ass:jump_nonsmoothness} as follows.
	\begin{enumerate}
		\item \cref{ass:jump_nonsmoothness} 1.\ holds by virtue of \cref{lem:criticality_implies_pred_bound}
		and \cref{rem:additional_terms}.
		\item \cref{ass:jump_nonsmoothness} 2.\ holds by virtue of \cref{lem:criticality_positive_case_distinction}
		and \cref{rem:additional_terms}.
		\item \cref{ass:jump_nonsmoothness} 3.\ holds by virtue of \cref{lem:criticality_lipschitz}.
	\end{enumerate}
The assumptions of 
\cref{lem:criticality_implies_pred_bound,lem:criticality_positive_case_distinction,lem:criticality_lipschitz}
directly follow from the definition of $X$, $p > 1$ (\cref{lem:criticality_implies_pred_bound,lem:criticality_positive_case_distinction})
and $p = \infty$ (\cref{lem:criticality_lipschitz}), and \cref{ass:general_assumption}.
\end{proof}

\begin{lemma}\label{lem:criticality_implies_pred_bound}
	Let $\{\nabla F(w) \,:\, w \in X\}$ be uniformly bounded in $W^{1,p}(0,1)$ for some $p > 1$.
	Let $\underline{\Delta}_a(w)$ be the minimum distance of neighboring switches of opposite signs
	(or to the boundary). Then we obtain for all $\Delta \le \underline{\Delta}_a(w)$ that
	\[ \pred(w,\Delta) \ge c_0 \crit(w)\Delta - c_1 \Delta^{1 + \frac{p-1}{p}} \]
	holds with $c_0 \coloneqq (n_{\max}|\max W - \min W|)^{-1}$ and
	$c_1 \coloneqq \sup_{w \in X}\|\nabla F(w)'\|_{L^p(0,1)} \tfrac{p}{2p-1}$ 
	for all $w \in X$.
\end{lemma}
\begin{proof}
	Because of the continuous embedding $W^{1,p}(0,1) \hookrightarrow C([0,1])$, $\nabla F(w)$ is continuous.
	Let $g = \nabla F(w)$. We first observe
	\begin{align*}\label{eq:1norm_supnorm_estimate}
	n_{\max} \max_{i} |\nabla F(w)(t_i)||w(t_i^+) - w(t_i^-)| &\ge n(w) \max_{i} |\nabla F(w)(t_i)||w(t_i^+) - w(t_i^-)| \\
	&\ge \crit(w) \ge \max_{i} |\nabla F(w)(t_i)||w(t_i^+) - w(t_i^-)|
	\end{align*}
	and fix $i \in \argmax_j |\nabla F(w)(t_j)|$. We only analyze the situation $\nabla F(w)(t_i) < 0$ with
	$w(t_i^-) >  w(t_i^+)$ in detail here. The other finitely many situations (jump up instead of down or $\nabla F(w)(t_i) > 0$)
	follow with a symmetric argument; see also the considerations in Lemma 4.8
	in \cite{leyffer2022sequential}.
	
	Next, we observe that for all $h \le \underline{\Delta}_a(w)$ with
	\[ \underline{\Delta}_a(w) \coloneqq \min\Big\{ 1, \underbrace{\min\{ t_j\,:\, t_i < t_j \text{ and } w(t_j^+) > w(t_j^-)\}}_{\eqqcolon t_{i\text{ next}}}\Big\} - t_i,
	\]
	the function $w + d_h$, $d_h = \chi_{[t_j, t_j + h)}$, is feasible for \eqref{eq:q} and $\|d_h\|_{L^1} = h$ holds.
	This means that we can shift the downward jump at $t_i$ at $t_i$ to the right until we reach the right boundary
	of the domain $(0,1)$ or the next upward jump.
	
	Using that only downward jumps can occur between $t_i$ and $t_{i} + h$, we deduce $\TV(w) = \TV(w + d_h)$, which implies
	\begin{align}
	\pred(w, h) &\ge - \int_{t_i}^{t_i + h} g(s)\dd s \nonumber\\
	&= -h g(t_i) - \int_0^h \int_0^s g'(t_i + \sigma) \dd \sigma \dd s
	&&\text{$g \in W^{1,p}(0,1)$} \nonumber\\
	&\ge -h g(t_i) - \int_0^h\|g'(t_i + \cdot)\|_{L^p(0,s)}s^\frac{p - 1}{p} \dd s
	&&\text{H\"older's inequality} \nonumber\\
	&\ge - h g(t_i) - h^{\frac{2p - 1}{p}} \frac{p}{2p - 1}  \|g'\|_{L^p(0,1)} \nonumber\\
	&\ge \frac{1}{n_{\max}|\max W - \min W|}\crit(w) h - h h^{\frac{p - 1}{p}} \frac{p}{2p - 1}  \|g'\|_{L^p(0,1)} \label{eq:main_pred_estimate}
	\end{align}
	so that we obtain the claim with the assumed constants.
\end{proof}

\begin{lemma}\label{lem:criticality_positive_case_distinction}
	Let $F : L^1(0,1) \to \R$ be continuously differentiable on $X$,
	that is, $\|h_k\|_{L^1}^{-1}(F(x + h_k) - F(x) - (\nabla F(x),h_k)_{L^2}) \to 0$
	for all $x \in X$, $\|h_k\| \searrow 0$ with $x + h_k \in X$ for all $k \in \N$ and
	$\nabla F : L^1(0,1) \to L^\infty(0,1)$ be Lipschitz continuous on $X$ with Lipschitz constant $\kappa > 0$.
	Let $\{\nabla F(w) \,:\, w \in X \}$ be uniformly bounded in $W^{1,p}(0,1)$ for some $p > 1$.	
	Then there exists $\underline{\Delta}_b > 0$ such that for
	all iterates $w_n$, $w_{n+1}$ produced by \cref{alg:trm} and
	\[ R_n \coloneqq (1 - \sigma)\pred(w_n,\Delta_n)  - |\ared(w_n,w_{n+1}) - \pred(w_n,\Delta_n)|,
	\]
	we obtain
	\begin{align*}
	R_n &\ge (1 - \sigma)\Big(c_0 \crit(w_n) \Delta_n - c_1 \Delta_n^{1 + \frac{p - 1}{p}}  - c_2 \Delta_n^2\Big)
	&&\text{if } \Delta_n \le \underline{\Delta}_a(w_n)\\
	\pred(w_n,\Delta_n) &\ge \delta 
	&&\text{if } \underline{\Delta}_a(w_n) \le \min\{\Delta_n,\underline{\Delta}_b\},\\
	R_n &\ge \delta
	&&\text{if } \underline{\Delta}_a(w_n) \le \Delta_n \le \underline{\Delta}_b
	\end{align*}
	for positive constants $c_0$, $c_1$ that are chosen as in \cref{lem:criticality_implies_pred_bound}, 
	$c_2 \coloneqq \frac{\kappa}{1 - \sigma}$, and $\delta \coloneqq \frac{1 - \sigma}{2}$.
\end{lemma}
\begin{proof}
	Let $\kappa > 0$ denote the Lipschitz constant of $\nabla F : L^1(0,1) \to L^\infty(0,1)$ on $X$.
	For iteration $n \in \N$, we deduce from the mean value theorem that there is $\xi_n = w_n + \tau_n (w_{n+1} - w_n)$
	for some $\tau_n \in [0,1]$ such that
	\begin{align*}
	\ared(w_n, w_{n+1})
	&=\sigma \pred(w_n, \Delta_n) + (1 - \sigma) \pred(w_n, \Delta_n) + (F(\xi_n) - \nabla F(w_n) , w_n - w_{n+1})_{L^2}\\
	&\ge \sigma \pred(w_n, \Delta_n) + \underbrace{(1 - \sigma) \pred(w_n, \Delta_n) - \kappa \|w_n - w_{n+1}\|^2}_{\eqqcolon R_n}.
	\end{align*}
	
	As in the proof of \cref{lem:criticality_implies_pred_bound}, we now fix $i \in \argmax_j |\nabla F(w_n)(t_j)|$ and
	again only analyze the situation $\nabla F(w_n)(t_i) < 0$ with $w_n(t_i^-) > w_n(t_i^+)$ in detail here.
	The other three cases follow with a symmetric argument.
	Analogously to \cref{lem:criticality_implies_pred_bound} and its proof, we deduce
	\begin{align*}
	R_n &\ge (1 - \sigma) 
	\big(c_0\crit(w_n) h - c_1 h h^{\frac{p - 1}{p}} -c_2  \|w_n - w_{n+1}\|_{L^1(0,1)}^2 \big)
	\end{align*}
	for $h \le \underline{\Delta}_a(w_n)$ and $h \le \Delta_n$.
	
	If $\Delta_n \le \underline{\Delta}_a(w_n)$, the choice $h = \Delta_n$  gives
	\begin{gather}\label{eq:case_two_positivity}
	R_n \ge c_0 \crit(w_n) \Delta_n - c_1 \Delta_n \Delta_n^{\frac{p - 1}{p}} - c_2\Delta_n^2.
	\end{gather}
	%
	
	If $\underline{\Delta}_a(w_n) \le \Delta_n$ holds,
	then the choice $d_n = \chi_{[t_j, t_j + \underline{\Delta}_a(w_n))}$ implies that $w_n + d_n$
	is feasible for $\text{{\ref{eq:tr}}}(w_n, \nabla F(w_n), \Delta_n)$ and gives
	\begin{align*}
	R_n \ge (1 - \sigma)
	\Big(1 - G \underline{\Delta}_a(w_n) - \frac{L}{1 - \sigma}\underline{\Delta}_a(w_n)^2\Big)
	\ge (1- \sigma)\Big(1 - G \Delta_n - \frac{\kappa}{1 - \sigma}\Delta_n^2\Big) 
	\end{align*}
	because $\TV(w_n) \ge \TV(w_n + d_n) + 1$ holds. Consequently, there exists $\underline{\Delta}_0 > 0$ such that if
	\[ \underline{\Delta}_a(w_n) \le \Delta_{n} \le \underline{\Delta}_0, \]
	we have $R_n \ge \frac{1 - \sigma}{2}$. 
	
	Similarly, we obtain for $\underline{\Delta}_a(w_n) \le \min\{\Delta_n,\underline{\Delta}_0\}$ that
	\begin{align*}
	\pred(w_n,\Delta_n) \ge \pred(w_n,\underline{\Delta}_a(w_n))
	\ge 1 - G \underline{\Delta}_a(w_n) \ge \frac{1}{2} \ge \frac{1 - \sigma}{2},
	\end{align*}
	where the first inequality is due to the monotonicity of $\pred(w_n,\cdot)$.
	
	The claim follows with $\underline{\Delta}_b \coloneqq \underline{\Delta}_0$.
\end{proof}

\begin{lemma}\label{lem:criticality_lipschitz}
Let $F : L^1(0,1) \to \R$ be continuously differentiable on $X$
and $\nabla F : L^1(0,1) \to L^\infty(0,1)$ be Lipschitz continuous on $X$.
Let $\{\nabla F(w) \,:\, w \in X \}$ be uniformly bounded in $W^{1,\infty}(0,1)$.
Let $\delta$ be as in \cref{lem:criticality_positive_case_distinction}.
There exist $\underline{\Delta}_c > 0$ and $L > 0$ such that $d(x_n,x_{n+1}) \le \Delta_n \le \underline{\Delta}_c$
for two subsequent iterates produced by \cref{alg:trm} implies
\[ \pred(w_n,\Delta_{n+1}) \ge \delta \quad\text{or}\quad |\crit(w_n) - \crit(w_{n+1})| \le L \|w_n - w_{n+1}\|_{L^1}. \]
\end{lemma}
\begin{proof}
Let $G \coloneqq  \sup_{w \in X} \|\nabla F(w)\|_{L^\infty(0,1)}$. Then we observe
\begin{align*}
\pred(w_n,\Delta_n)
&= \TV(w_n) - \TV(w_{n+1}) + (\nabla F(w_n), w_n - w_{n+1})_{L^2} \\
&\left\{
\begin{aligned}
\le \TV(w_n) - \TV(w_{n+1}) + G \Delta_n,\\
\ge \TV(w_n) - \TV(w_{n+1}) - G \Delta_n.
\end{aligned}
\right.
\end{align*}
Consequently, a choice $\underline{\Delta}_c \le 0.5 G^{-1}$ implies that the predicted reduction
would be negative if $\TV(w_{n+1}) > \TV(w_n)$ holds, which is not possible.
Therefore, $\TV(w_n) \ge \TV(w_{n+1})$ has to hold. If $\TV(w_n) > \TV(w_{n+1})$, then we obtain 
\[ \pred(w_n,\Delta_n) \ge 0.5 \ge \delta \]
if $\underline{\Delta}_c \coloneqq 0.5 G^{-1}$ because $\TV(w_n) - \TV(w_{n+1}) \in \Z$.

It remains to analyze the case $\TV(w_n) = \TV(w_{n+1})$ under the condition that
$d(w_n,w_{n+1}) \le \Delta_n \le 0.5 G^{-1}$. In this case, we need to overestimate the term
\begin{multline*}
|\crit(w_n) - \crit(w_{n+1})|
=\\ \bigg|\sum_{i=1}^{n(w_n)} |\nabla F(w_n)(t_i)||{w_n}(t^+_i) - {w_n}(t^-_i)| 
 - \sum_{j=1}^{n(w_{n+1})} |\nabla F(w_{n+1})(t_j)||{w_{n+1}}(t^+_j) - {w_{n+1}}(t^-_j)|\bigg|
\end{multline*}
We now infer that there is a relation between the jumps at $t_i$ with jump heights ${w_n}(t^+_i) - {w_n}(t^-_i)$
of $w_n$ and the jumps at $t_j$ with jump heights ${w_{n+1}}(t^+_j) - {w_{n+1}}(t^-_j)$ of $w_{n+1}$.

First we always cluster sequences of jumps whose jump heights have the same sign together along the
interval $(0,1)$ so that we obtain an alternating pattern of positive and negative jump heights for both $w_n$ and $w_{n+1}$.
The order (positive/negative) of these two patterns have to coincide as well as the sum of the jump heights inside corresponding
clusters. Otherwise, we can again show that $w_{n+1}$ was not optimal for the trust-region subproblem due to
$\|w_n - w_{n+1}\|_{L^1} \le \Delta_n\le 0.5 G^{-1}$ with a similar argument as above.

Next, we map the switching points to each other inside corresponding clusters in the following way.
We repeat each switching point $t_i$ of $w_n$ inside a cluster $|{w_n}(t^+_i) - {w_n}(t^-_i)|$ times
and similar for each switching point of $t_j$ of $w_{n+1}$.
Now, we can compute a one-to-one assignment (bipartite perfect matching) of these switches between the (increased) corresponding
clusters such that
\[ \|w_n - w_{n+1}\|_{L^1([\min \{t_i,t_j\},\max \{t_i,t_j\}))} \ge |t_i - t_j| \]
holds for matched switches $t_i$ and $t_j$ by construction of the switches from $w_n$ and $w_{n+1}$.

As a consequence, we obtain
\begin{multline*}
\bigg|\sum_{i=1}^{n(w_n)} |\nabla F(w_n)(t_i)||{w_n}(t^+_i) - {w_n}(t^-_i)| - \sum_{j=1}^{n(w_{n+1})} |\nabla F(w_{n+1})(t_j)||{w_{n+1}}(t^+_j) - {w_{n+1}}(t^-_j)|\bigg|
\\
\le L \|w_n - w_{n+1}\|_{L^1(0,1)}
\end{multline*}
with the choice $L \coloneqq (G  + \sup\{\|\nabla F(w)\|_{W^{1,\infty}(0,1)} : w \in X\}) n_{\max}|\max W - \min W|$, which proves the claim.
\end{proof}

\section{Computational experiments}\label{sec:numerics}
We describe our computational experiments in \cref{sec:description} and provide the results in \cref{sec:results}

\subsection{Experiment description}\label{sec:description}
In order to assess the effect of the different update strategies for the trust-region radius on the runtime performance,
we use two instances of \eqref{eq:q} as benchmark problems, where we scale the $\TV$-term in the objective by different 
scalars $\alpha > 0$ since the scaling has a strong influence on the numbers of required iterations in practice and thus 
the overall runtime of the algorithm. There is no theoretical change because this scaling is equivalent to scaling $F$ by $\nicefrac{1}{\alpha}$.
In particular, the problems are generally computationally less expensive for relatively small and 
relatively large values of $\alpha$ and have a runtime peak for intermediate values of $\alpha$,
see, e.g., \cite{baraldi2024domain}.
Then we execute \cref{alg:trm} as well as Alg.\ 1 from \cite{leyffer2022sequential}, where we note that we have implemented
them identically so that the only difference is the different behavior of the trust-region update on acceptance of a step
(doubling in \cref{alg:trm} and reset to some finite $\Delta_{\max}$ in Alg.\ 1 from \cite{leyffer2022sequential}).
For the subproblem solver, we use our most recent implementation of the topological sorting-based approach described
in \cite{severitt2023efficient}. 

In line with the naming in \cite{leyffer2022sequential,manns2023on},
we denote \cref{alg:trm} (without trust-region radius reset)
by SLIP-NR and Alg.\ 1 from \cite{leyffer2022sequential} (with
trust-region reset) by SLIP-RT.

The first benchmark problem is taken from \cite{severitt2023efficient} and is an integer optimal control problem that
is governed by a steady heat equation on a one-dimensional domain (an interval)
with $W = \{-2,\ldots,23\}$. It is described in detail in Section 5.1 in
\cite{severitt2023efficient} and we have discretized the PDE, its adjoint, and the objective using \texttt{DOLFINx} 0.9.0 \cite{barrata2023dolfinx}.
Specifically, we have discretized the input variable using a piecewise constant ansatz and the state variable using a continuous Lagrange
order 1 ansatz on a uniform grid of $N = 4096$ intervals that discretize the computational domain.

The second benchmark problem is a one-dimensional signal reconstruction problem
with $W = \{-2,\ldots,2\}$ that is described in Section 5 in \cite{leyffer2022sequential}.
As in \cite{leyffer2022sequential}, we use a piecewise constant ansatz for the input variable and a Legendre--Gauss quadrature of order 5 per interval
for the discretization of the convolution operator that occurs therein. Slightly different to the setting in \cite{leyffer2022sequential},
we choose $f(t) = 0.2 \cos(2(t - 1)\pi - 0.25)\exp(t - 1)$ for $t \in (-1,1)$. Again, we discretize the domain uniformly into $N = 4096$ intervals.

We execute algorithm variants for both problems for the scalings $\alpha \in \{10^{-6}, 5\times 10^{-6}, 10^{-5}, 5\times 10^{-5}, 10^{-4}, 5\times 10^{-4}, 10^{-3}\}$. The algorithm stops when no progress can be made on the current
discretization as is indicated by a contraction of the trust-region radius
below the mesh size. In all experiments, the algorithms were initialized with the 
constant zero function. All experiments are carried out on a laptop computer with
an Intel(R) Core i7(TM) CPU with eight cores that is clocked at 2.5 GHz and has
64 GB RAM.

\subsection{Results}\label{sec:results}
SLIP-NR is significantly faster than SLIP-RT on all instances. The relative
speed-up is often (including on the slowest
instances per benchmark class) higher than 60\,\%.
Compared to this, the achieved resulting objective values are most often
comparable with objective value differences less than 5\,\%. In most cases,
and in particular if the difference is significant, SLIP-RT achieves a better
final objective value.

For the steady heat equation benchmark problem, for the two computationally least 
expensive instances, SLIP-RT achieves objective values
that are 6.1\,\% and 74\,\% better than SLIP-NR. Detailed results are 
provided in \cref{tab:exp_pde}. To give a visual impression, we 
show the resulting controls for $\alpha = 5\times 10^{-6}$ in \cref{fig:ctrl_visualization}. For the signal reconstruction benchmark,
for the two computationally most expensive instances
of SLIP-RT, SLIP-RT achieves objective values
that are 8.9\,\% and 8.0\,\% better than SLIP-NR.
Detailed results are provided in \cref{tab:exp_conv}.

\begin{table}[ht]
	\caption{Runtimes in seconds for SLIP-RT ($t_{RT}$) and
	SLIP-NR ($t_{NR}$), relative runtime improvement of SLIP-NR, and
	resulting objective values ($J(x_{RT})$ and $J(x_{NR})$)
	for the steady heat equation benchmark.
	Significantly smaller ($>\,5\%$) runtime and objective values
	are highlighted bold-faced.}\label{tab:exp_pde}	
	\centering
	\begin{adjustbox}{width=\textwidth}	    
		\begin{tabular}{r|lllllllll}
			\toprule
			$\alpha\cdot 10^{-6}$ & $t_{RT}$ & $t_{NR}$ & $\frac{t_{RT} - t_{NR}}{t_{RT}}$ & $J(x_{RT})$ & $J(x_{NR})$ \\
			\midrule
			$1$             & $9.054\times 10^2$ & $\mathbf{2.955\times 10^2}$ & $67.4\;\%$ & $9.588\times 10^{-2}$ & $9.610\times 10^{-2}$ \\
			$5$             & $7.830\times 10^2$ & $\mathbf{2.938\times 10^2}$ & $66.3\;\%$ & $9.638\times 10^{-2}$ & $9.629\times 10^{-2}$ \\
			$1 \times 10^1$ & $1.037\times 10^3$ & $\mathbf{3.108\times 10^2}$ & $70.0\;\%$ & $9.637\times 10^{-2}$ & $9.653\times 10^{-2}$ \\
			$5 \times 10^1$ & $1.010\times 10^3$ & $\mathbf{3.055\times 10^2}$ & $69.8\;\%$ & $9.842\times 10^{-2}$ & $9.881\times 10^{-2}$ \\
			$1 \times 10^2$ & $7.104\times 10^2$ & $\mathbf{2.993\times 10^2}$ & $57.9\;\%$ & $1.010\times 10^{-1}$ & $1.016\times 10^{-1}$ \\
			$5 \times 10^2$ & $3.101\times 10^2$ & $\mathbf{1.761\times 10^2}$ & $43.2\;\%$ & $\mathbf{1.498\times 10^{-1}}$ & $1.589\times 10^{-1}$ \\
			$1 \times 10^3$ & $2.993\times 10^2$ & $\mathbf{5.204\times 10^1}$ & $82.6\;\%$ & $\mathbf{2.142\times 10^{-1}}$ & $3.727\times 10^{-1}$ \\
			\bottomrule
		\end{tabular}
	\end{adjustbox}
\end{table}

\begin{table}[ht]
	\caption{Runtimes in seconds for SLIP-RT ($t_{RT}$) and
		SLIP-NR ($t_{NR}$), relative runtime improvement
		of SLIP-NR, and resulting objective values
		($J(x_{RT})$ and $J(x_{NR})$)
		for the signal reconstruction benchmark.
		Significantly smaller ($>\,5\%$) runtime and objective values
		are highlighted bold-faced.}\label{tab:exp_conv}	
	\centering
	\begin{adjustbox}{width=\textwidth}	    
		\begin{tabular}{rlllllllll}
			\toprule
			$\alpha\cdot 10^{-6}$ & $t_{RT}$ & $t_{NR}$ & $\frac{t_{RT} - t_{NR}}{t_{RT}}$ & $J(x_{RT})$ & $J(x_{NR})$ \\
			\midrule
			$1$             & $5.380\times 10^1$ & $\mathbf{2.120\times 10^1}$ & $60.6\;\%$ & $2.129\times 10^{-4}$ & $2.170\times 10^{-4}$ \\
		    $5$             & $5.698\times 10^1$ & $\mathbf{2.234\times 10^1}$ & $60.8\;\%$ & $\mathbf{4.022\times 10^{-4}}$ & $4.378\times 10^{-4}$ \\			
			$1 \times 10^1$ & $5.404\times 10^1$ & $\mathbf{2.263\times 10^1}$ & $58.1\;\%$ & $\mathbf{5.850\times 10^{-4}}$ & $6.308\times 10^{-4}$ \\
			$5 \times 10^1$ & $4.563\times 10^1$ & $\mathbf{3.342\times 10^1}$ & $26.8\;\%$ & $1.541\times 10^{-3}$ & $1.521\times 10^{-3}$ \\
			$1 \times 10^2$ & $4.488\times 10^1$ & $\mathbf{1.392\times 10^1}$ & $69.0\;\%$ & $2.445\times 10^{-3}$ & $2.448\times 10^{-3}$ \\
			$5 \times 10^2$ & $1.315\times 10^1$ & $\mathbf{7.846}$            & $40.3\;\%$ & $6.074\times 10^{-3}$ & $6.074\times 10^{-3}$ \\
			$1 \times 10^3$ & $4.410$            & $\mathbf{2.470}$            & $44.0\;\%$ & $9.787\times 10^{-3}$ & $9.787\times 10^{-3}$ \\					
			\bottomrule
		\end{tabular}
	\end{adjustbox}
\end{table}

\begin{figure}[t]
	\vspace{-1cm}
	\centering
	\begin{subfigure}[b]{0.49\textwidth}
	\centering
	\includegraphics[width=.8\textwidth]{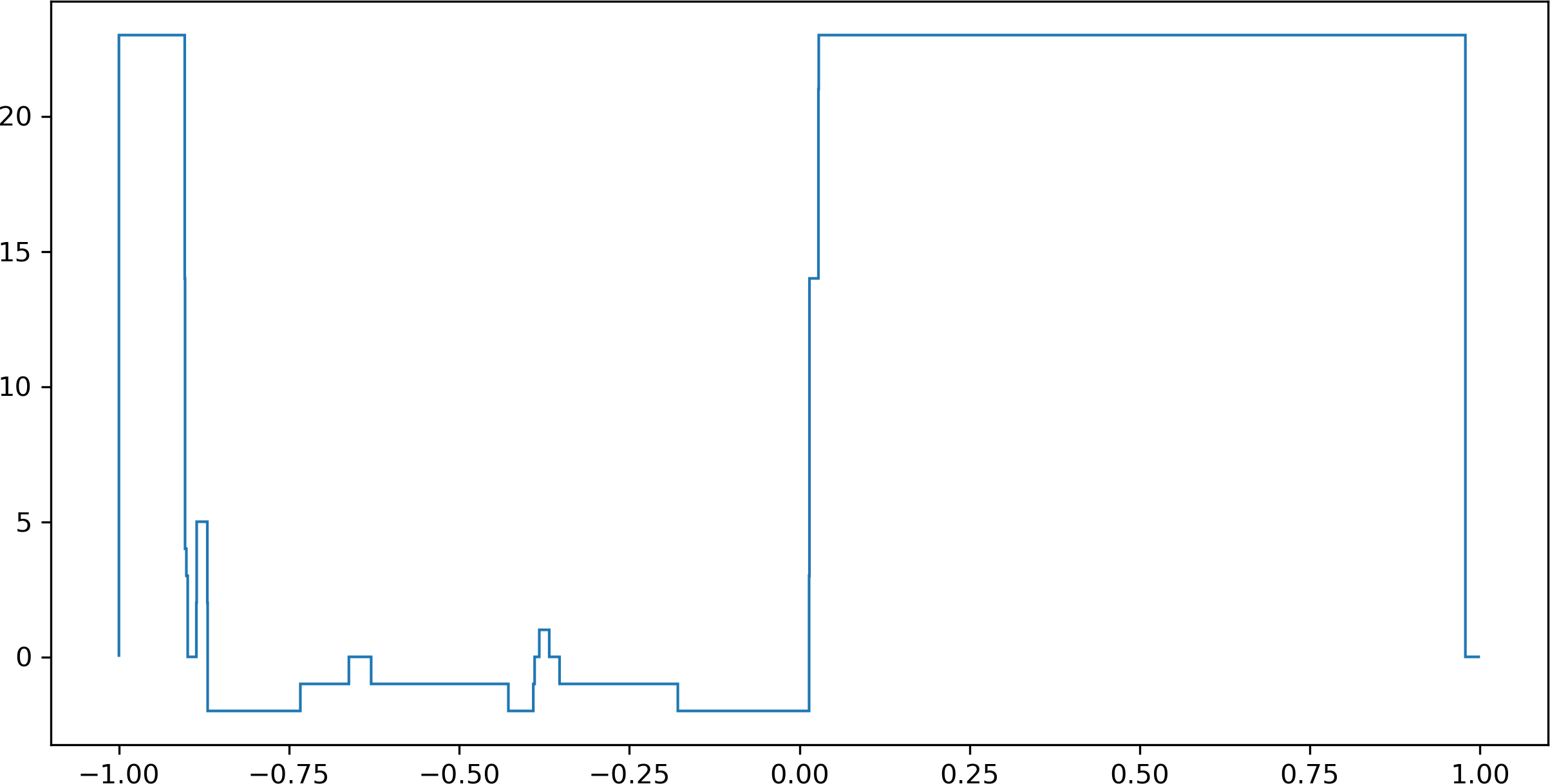}
	\caption{{\small SLIP-NR, $\alpha = 1\times 10^{-6}$}}
	\end{subfigure}
	\hfill
	\begin{subfigure}[b]{0.49\textwidth}  
	\centering 
	\includegraphics[width=.8\textwidth]{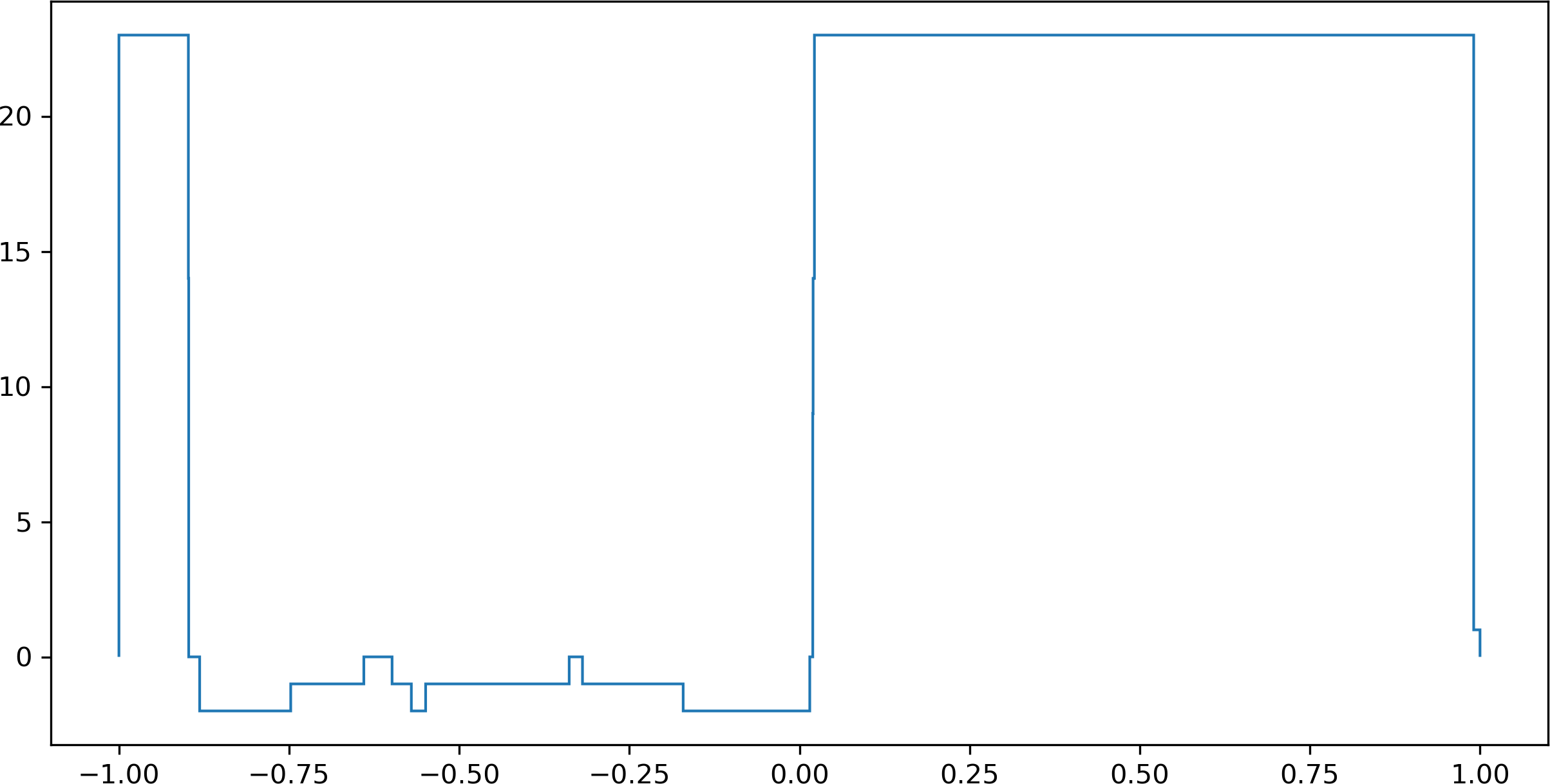}
	\caption{{\small SLIP-RT, $\alpha = 1\times 10^{-6}$}}
	\end{subfigure}
	\vspace{2mm}
	
	\begin{subfigure}[b]{0.49\textwidth}
	\centering
	\includegraphics[width=.8\textwidth]{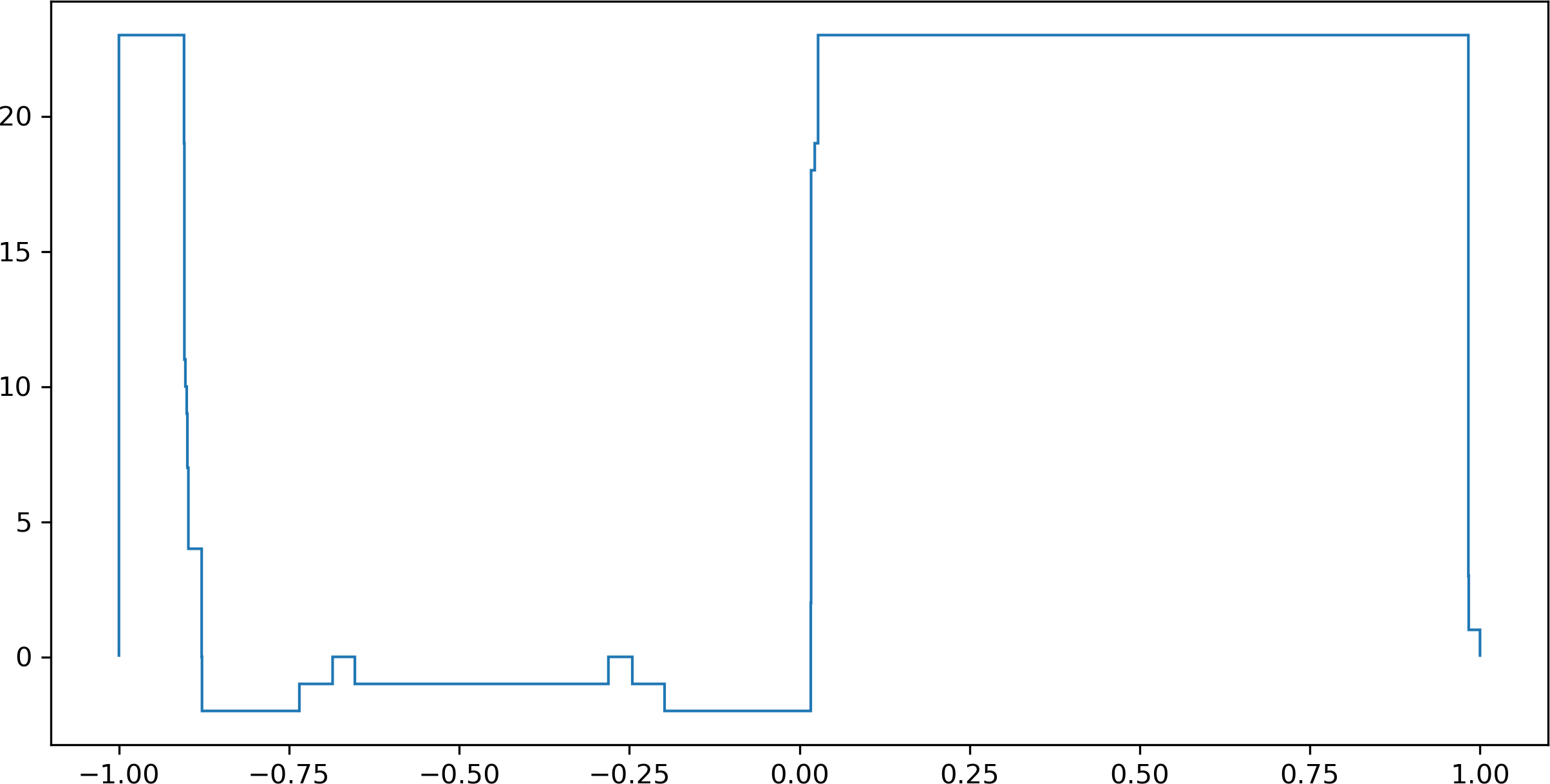}
	\caption*{{\scriptsize SLIP-NR, $\alpha = 5\times 10^{-6}$}}
	\end{subfigure}
	\hfill
	\begin{subfigure}[b]{0.49\textwidth}  
	\centering 
	\includegraphics[width=.8\textwidth]{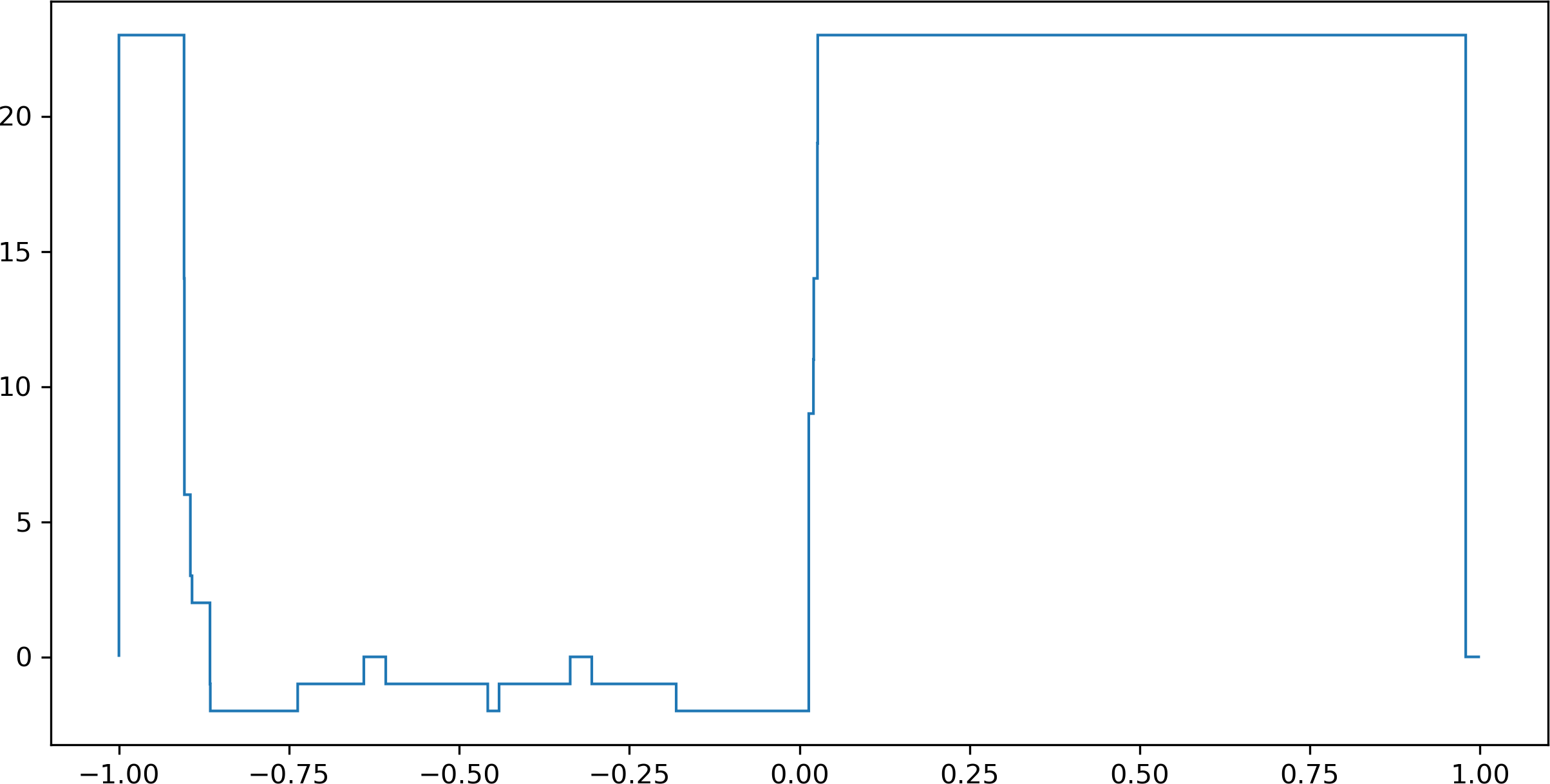}
	\caption*{{\scriptsize SLIP-RT, $\alpha = 5\times 10^{-6}$}}
	\end{subfigure}
	\vspace{2mm}

	\begin{subfigure}[b]{0.49\textwidth}
	\centering
	\includegraphics[width=.8\textwidth]{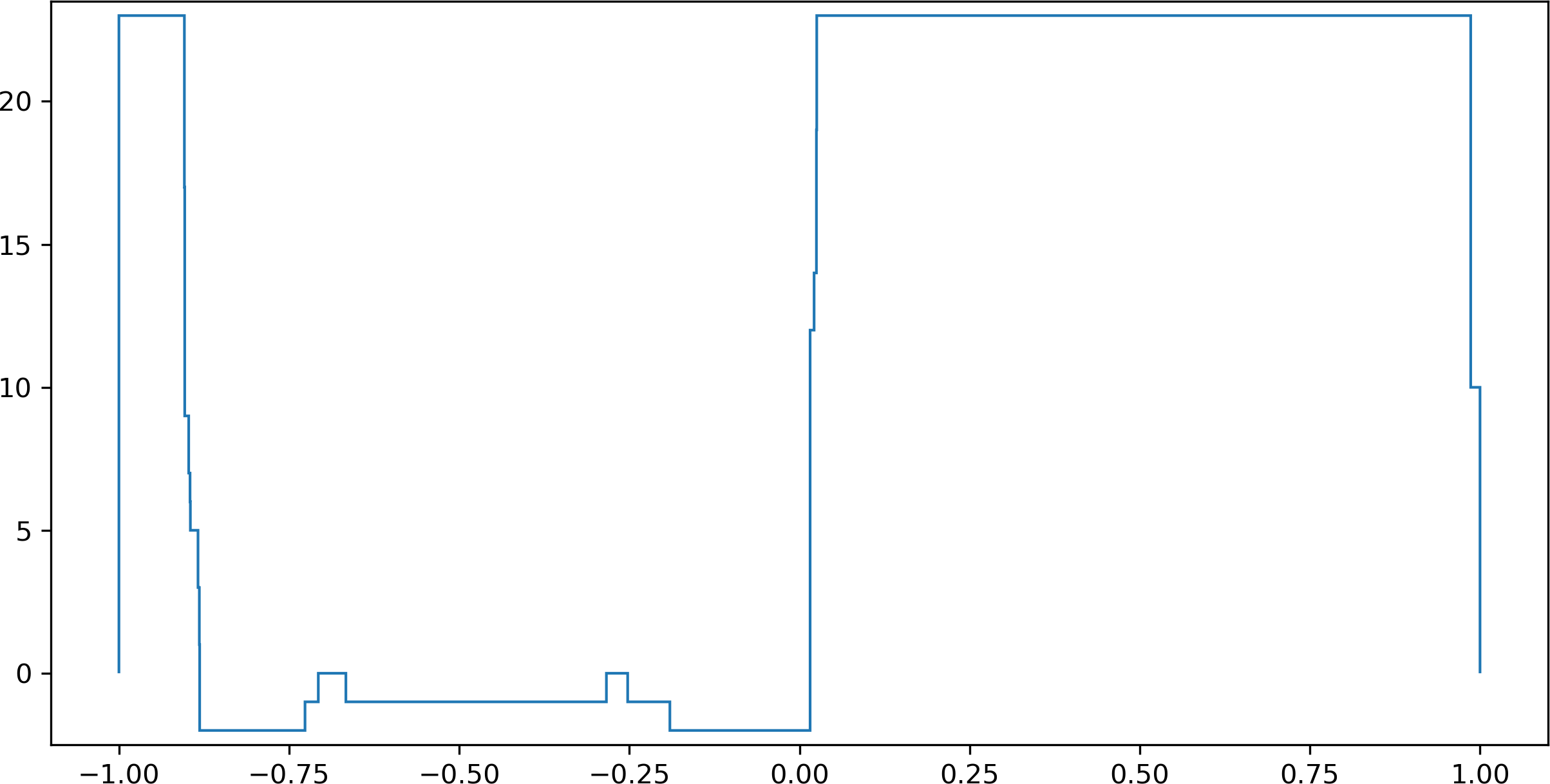}
	\caption*{{\scriptsize SLIP-NR, $\alpha = 1\times 10^{-5}$}}
	\end{subfigure}
	\hfill
	\begin{subfigure}[b]{0.49\textwidth}  
	\centering 
	\includegraphics[width=.8\textwidth]{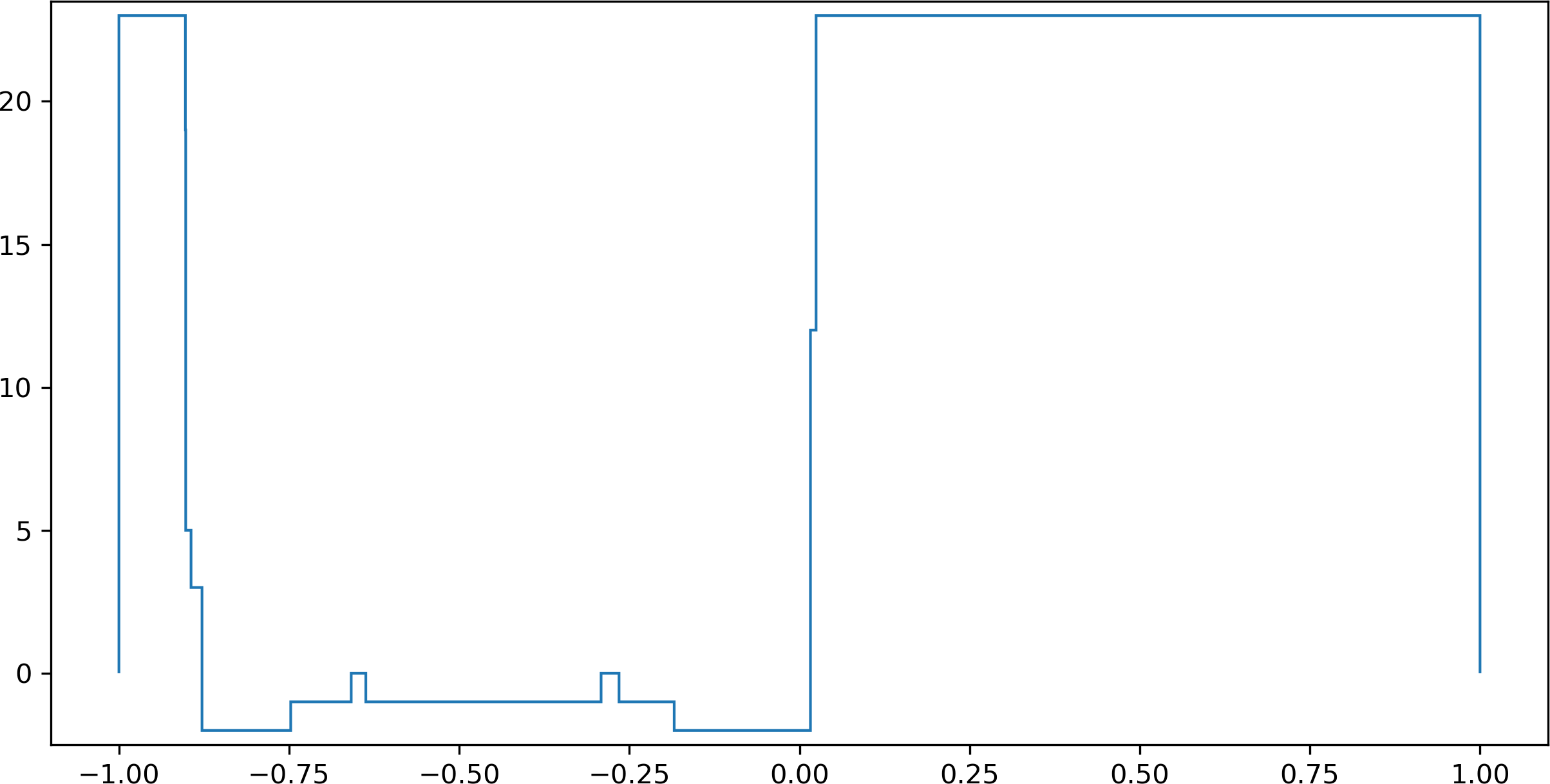}
	\caption*{{\scriptsize SLIP-RT, $\alpha = 1\times 10^{-5}$}}
	\end{subfigure}
	\vspace{2mm}

	\begin{subfigure}[b]{0.49\textwidth}
	\centering
	\includegraphics[width=.8\textwidth]{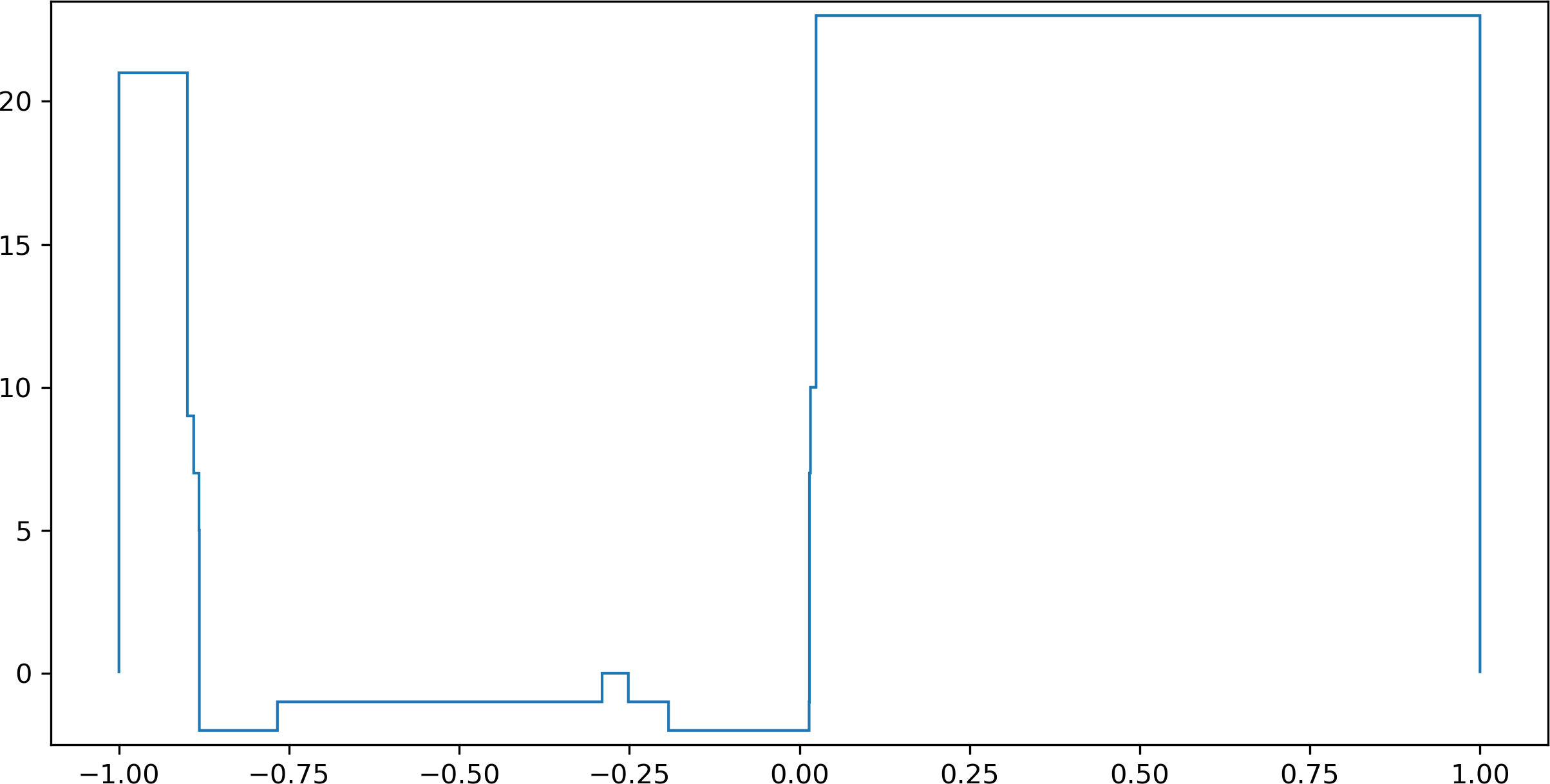}
	\caption*{{\scriptsize SLIP-NR, $\alpha = 5\times 10^{-5}$}}
	\end{subfigure}
	\hfill
	\begin{subfigure}[b]{0.49\textwidth}  
	\centering 
	\includegraphics[width=.8\textwidth]{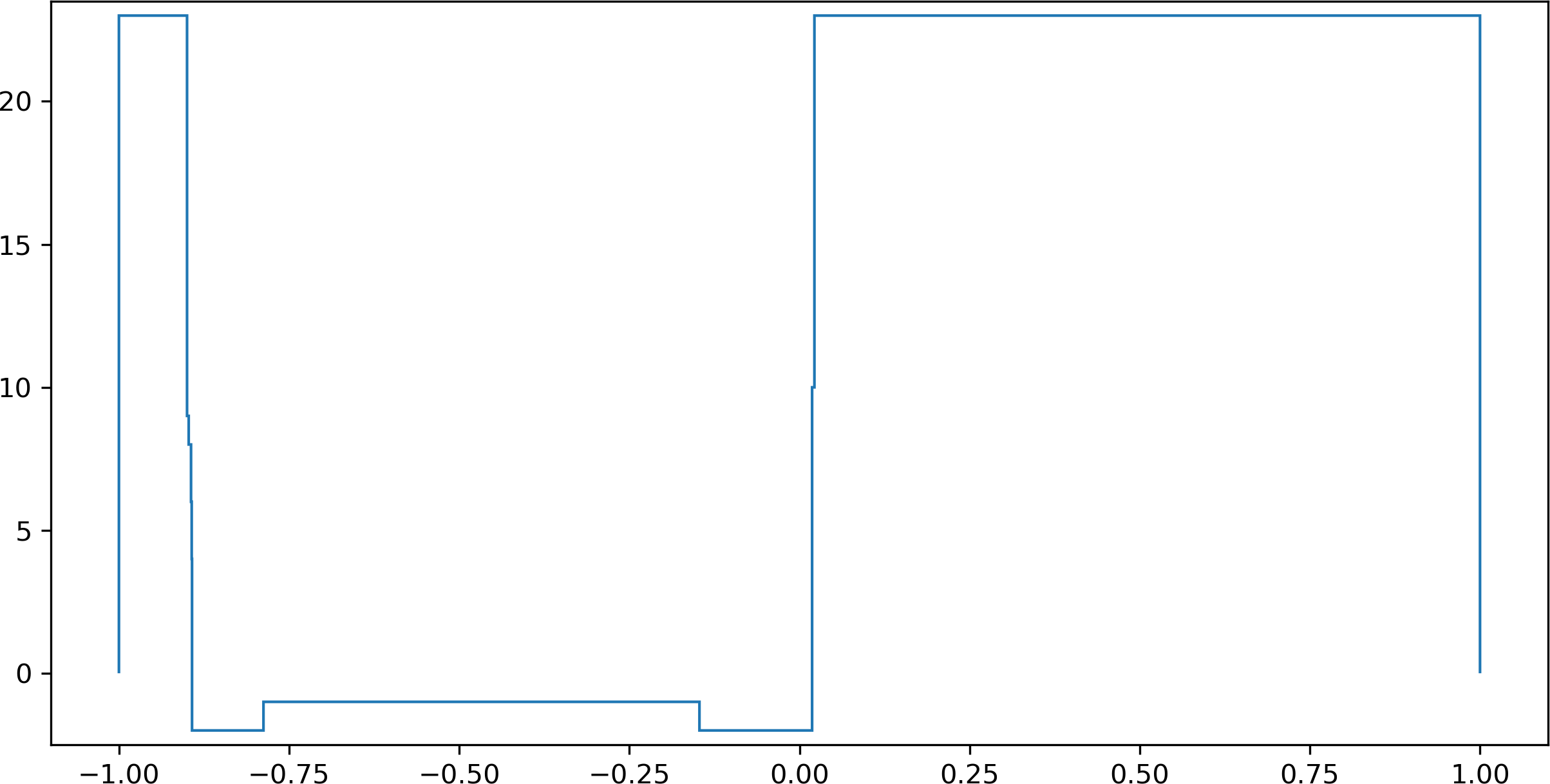}
	\caption*{{\scriptsize SLIP-RT, $\alpha = 5\times 10^{-5}$}}
	\end{subfigure}
	\vspace{2mm}

	\begin{subfigure}[b]{0.49\textwidth}
	\centering
	\includegraphics[width=.8\textwidth]{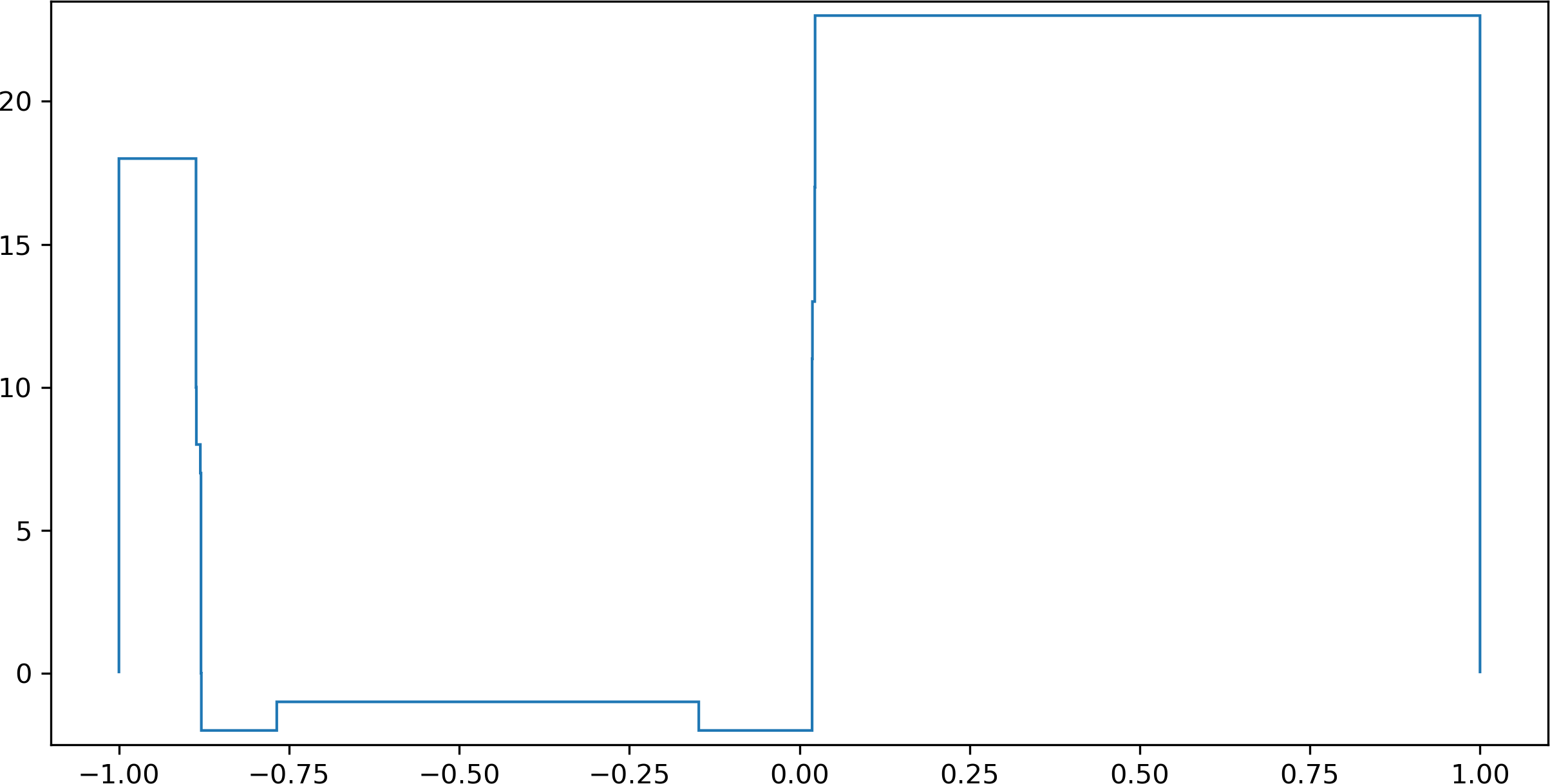}
	\caption*{{\scriptsize SLIP-NR, $\alpha = 1\times 10^{-4}$}}
	\end{subfigure}
	\hfill
	\begin{subfigure}[b]{0.49\textwidth}  
	\centering 
	\includegraphics[width=.8\textwidth]{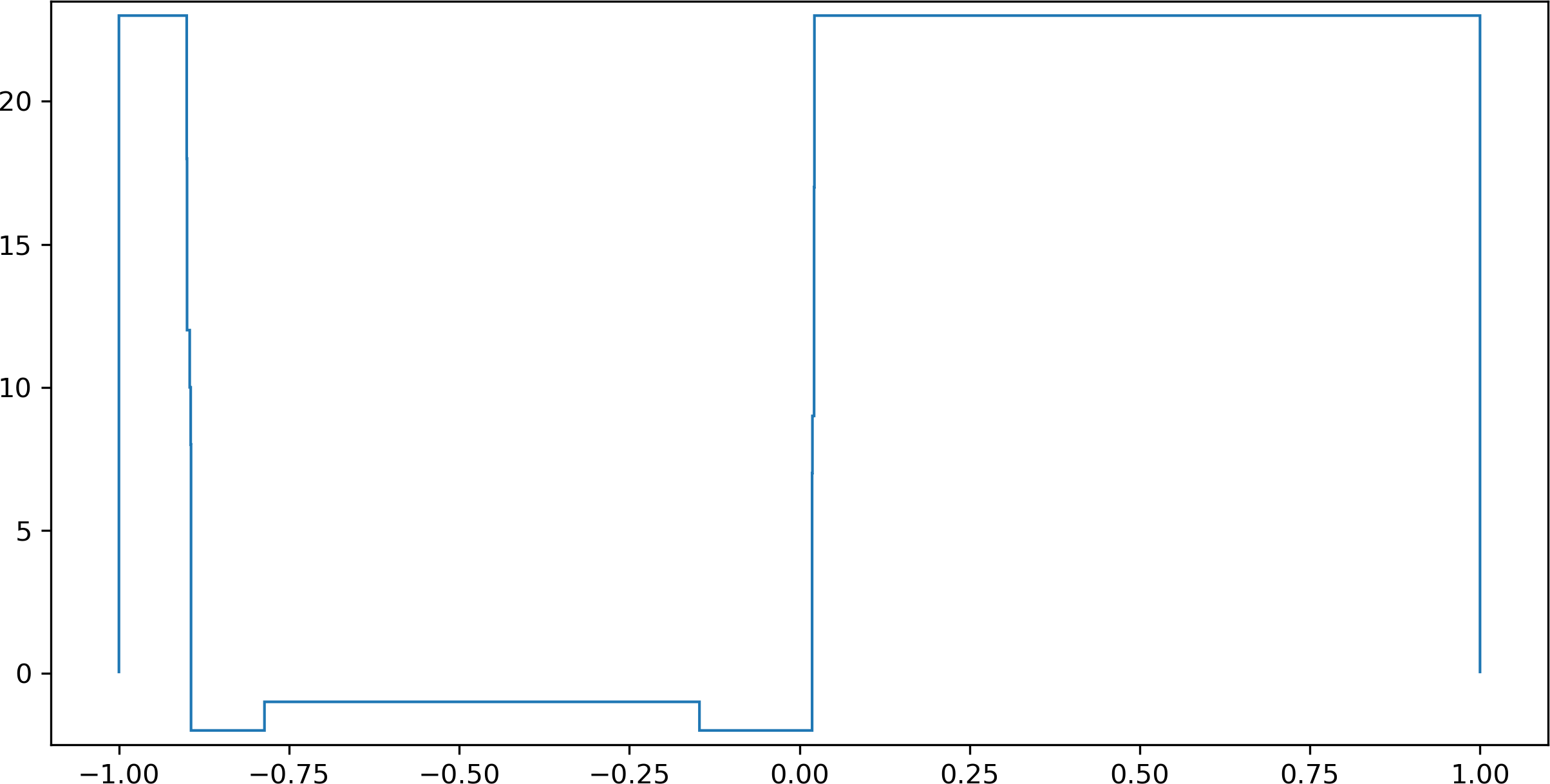}
	\caption*{{\scriptsize SLIP-RT, $\alpha = 1\times 10^{-4}$}}
	\end{subfigure}	
	\vspace{2mm}
	
	\begin{subfigure}[b]{0.49\textwidth}
	\centering
	\includegraphics[width=.8\textwidth]{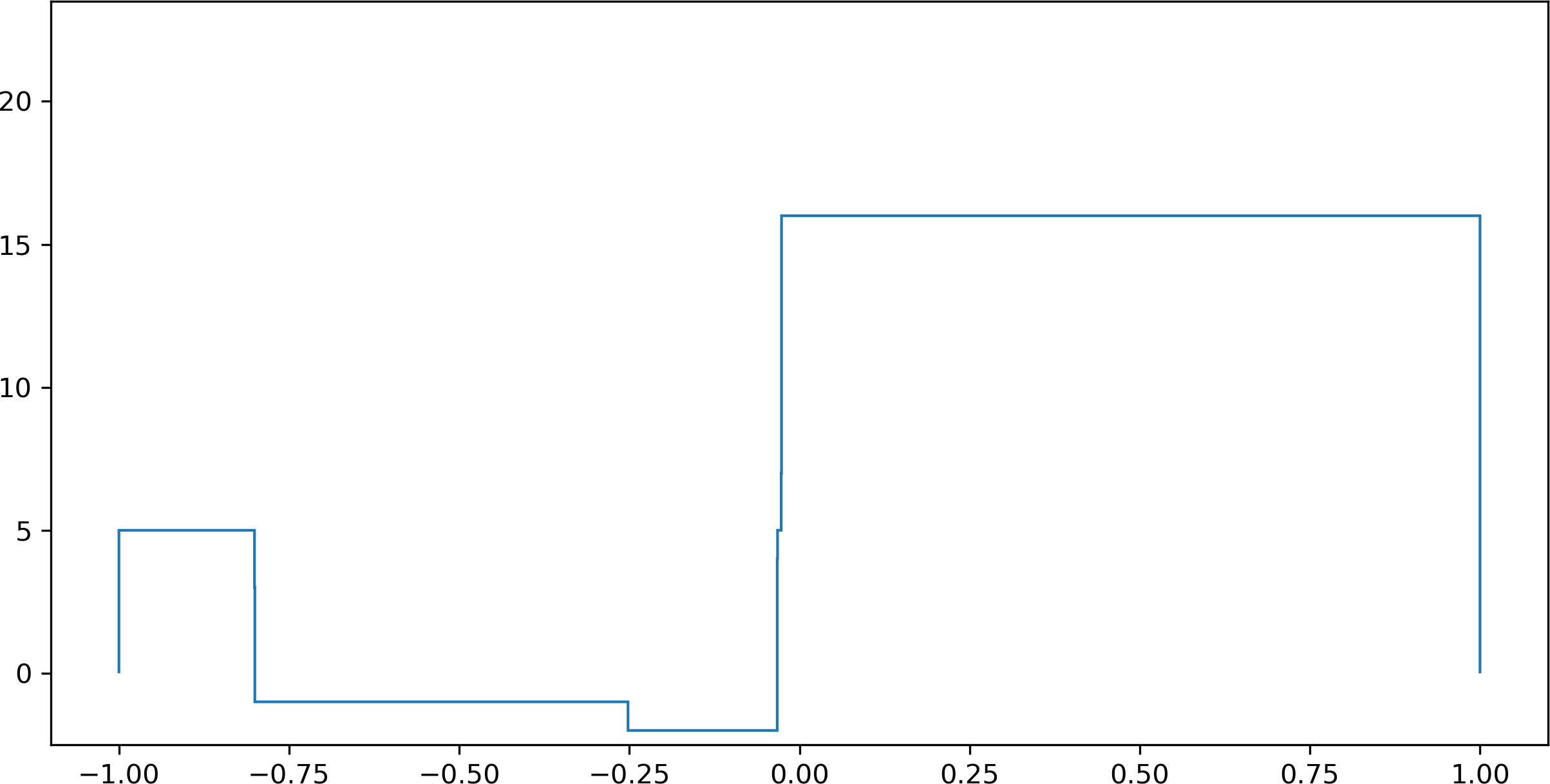}
	\caption*{{\scriptsize SLIP-NR, $\alpha = 5\times 10^{-4}$}}
	\end{subfigure}
	\hfill
	\begin{subfigure}[b]{0.49\textwidth}  
	\centering 
	\includegraphics[width=.8\textwidth]{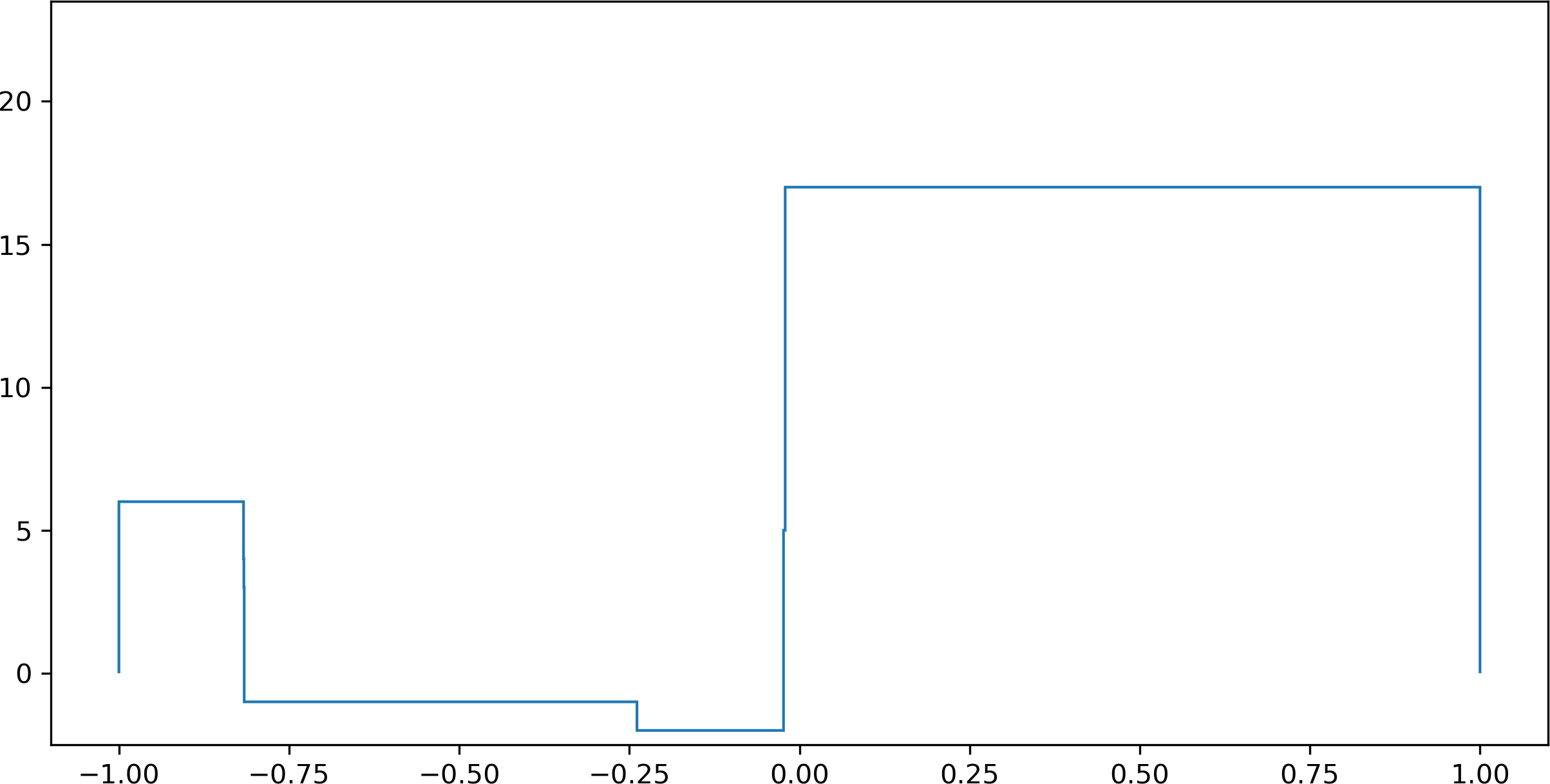}
	\caption*{{\scriptsize SLIP-RT, $\alpha = 5\times 10^{-4}$}}
	\end{subfigure}		
	\vspace{2mm}

	\begin{subfigure}[b]{0.49\textwidth}
	\centering
	\includegraphics[width=.8\textwidth]{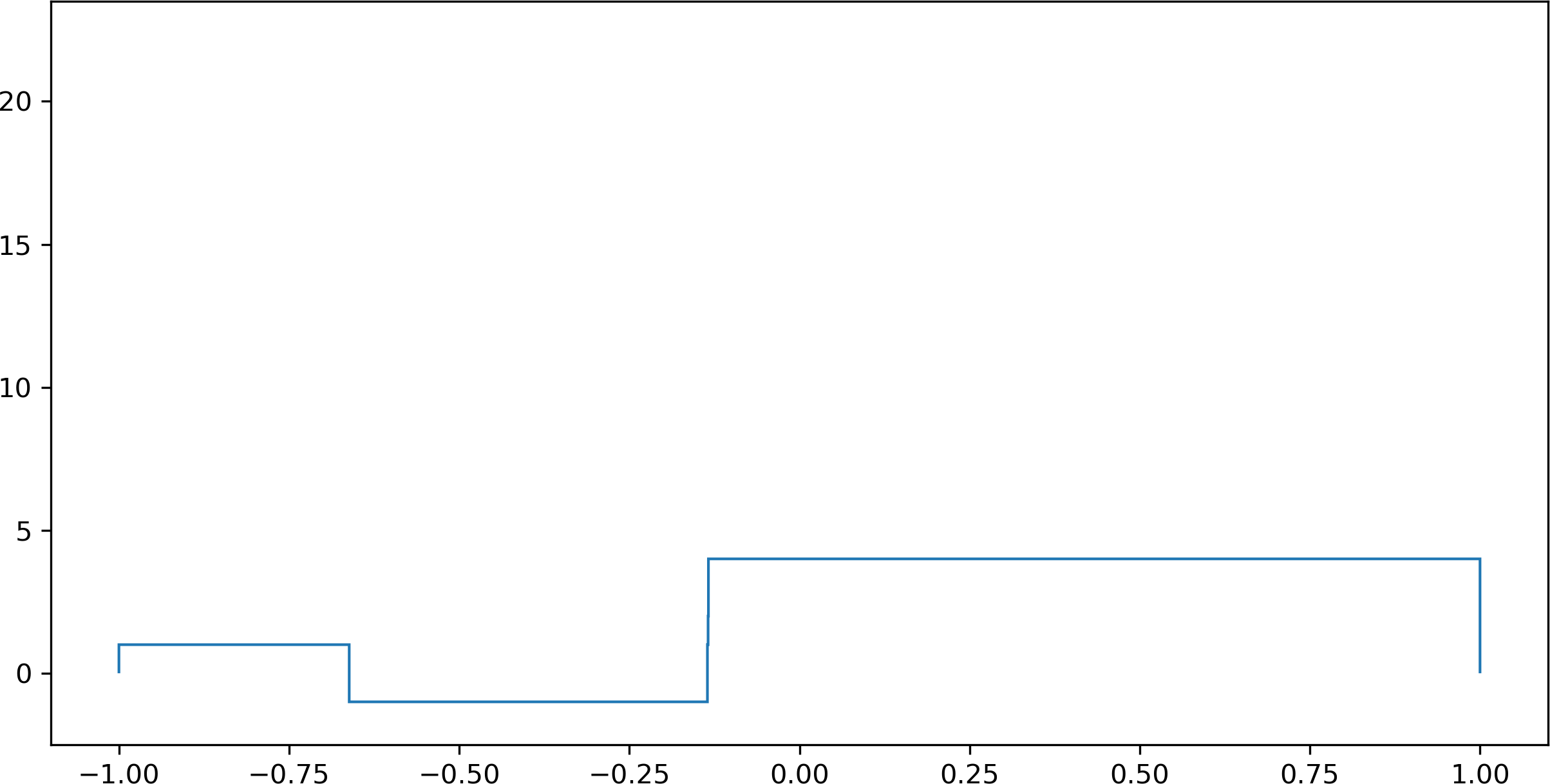}
	\caption*{{\scriptsize SLIP-NR, $\alpha = 1\times 10^{-3}$}}
	\end{subfigure}
	\hfill
	\begin{subfigure}[b]{0.49\textwidth}  
	\centering 
	\includegraphics[width=.8\textwidth]{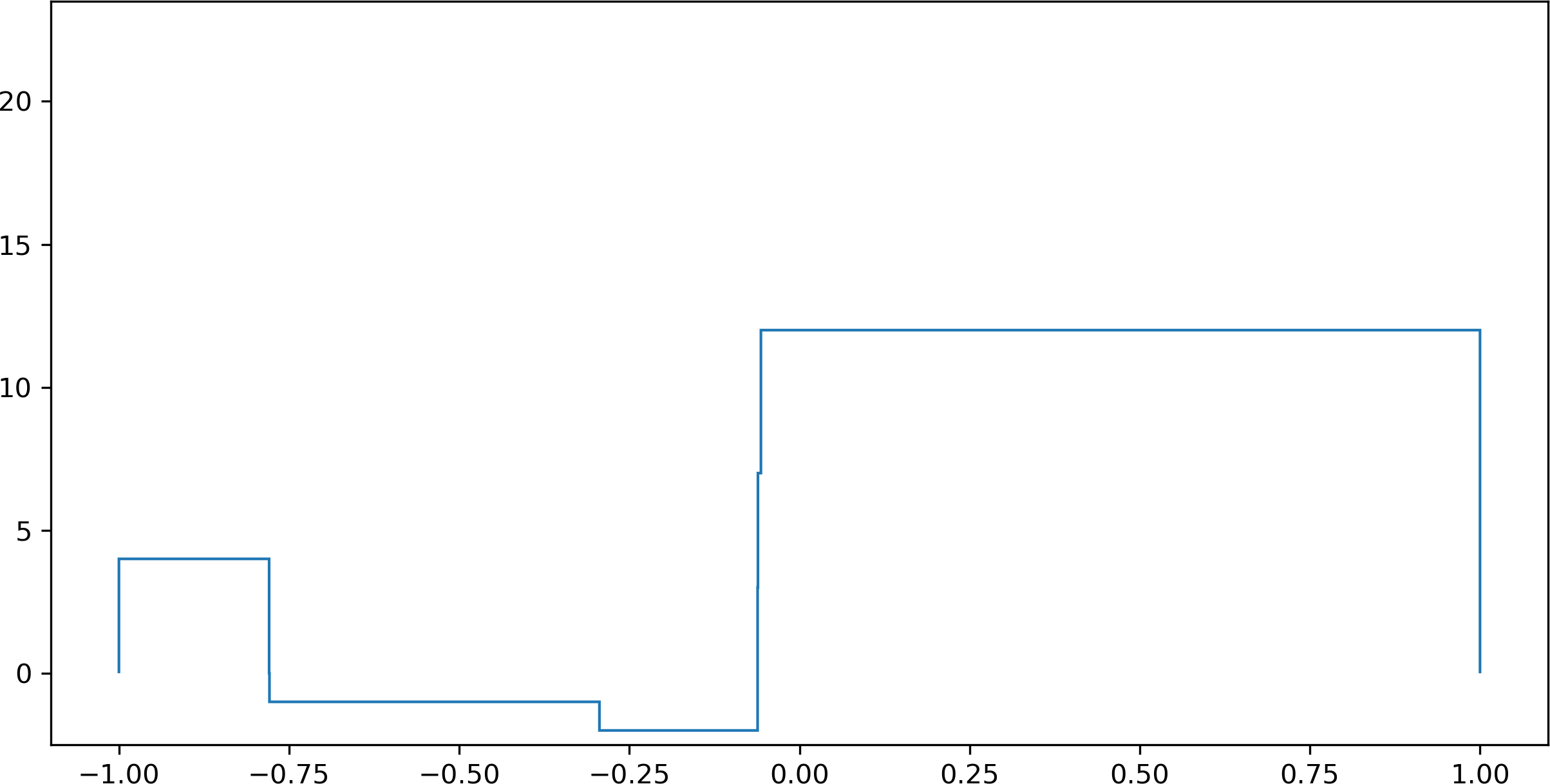}
	\caption*{{\scriptsize SLIP-RT, $\alpha = 1\times 10^{-3}$}}
	\end{subfigure}	

	\caption{Visualization of the controls produced by
	SLIP-NR (left) and SLIP-RT (right).}\label{fig:ctrl_visualization}
\end{figure}

\section{Conclusion}\label{sec:conclusion}
We have provided a convergence proof for a trust-region algorithm in
a metric space, where we have imposed assumptions that allow
for the convergence proof to work even when the trust-region radius
is not reset as is for example done in \cite{leyffer2022sequential}
but doubled upon acceptance of an iterate.

We have verified all of the imposed assumptions for a problem class
of integer optimal control problems on one-dimensional domains.
We have run the algorithm with and without the reset of the trust-region
radius on two (discretized) benchmark problems that fall into this problem
class. The achieved results show a substantial performance gain when avoiding
the reset of the trust-region radius on all instances and
of more than 50\,\% runtime reduction on more than half of the instances.
This runtime improvement comes at the cost that the points returned by SLIP
without trust-region radius reset have worse objective values in most
instances. The magnitude of this quality degradation is less than 10\,\%
on all but one of the considered benchmark instances.
While these observations may be coincidental,
the trust-region radius radius may occasionally yield acceptable points
that are farther away from the current iterate and thus lead to progression
towards stationary points with lower objective values. This heuristic idea
is key in the widely-used simulated annealing algorithm \cite{van1987simulated}.

\section*{Acknowledgment}
The author is grateful to Annika Schiemann (TU Dortmund University),
two anonymous referees, as well as Gerd Wachsmuth and Markus Friedemann
(both BTU Cottbus-Senftenberg) for helpful feedback on the manuscript.

\bibliographystyle{plain}
\bibliography{references}{}

\begin{thebibliography}{10}

\bibitem{adams2003sobolev}
Robert~A Adams and John~JF Fournier.
\newblock {\em Sobolev spaces}.
\newblock Elsevier, 2003.

\bibitem{ambrosio2000functions}
L.~Ambrosio, N.~Fusco, and D.~Pallara.
\newblock {\em Functions of bounded variation and free discontinuity problems},
  volume 254 of {\em Oxford Mathematical Monographs}.
\newblock Clarendon Press Oxford, 2000.

\bibitem{baraldi2024domain}
Robert Baraldi and Paul Manns.
\newblock Domain decomposition for integer optimal control with total variation
  regularization.
\newblock {\em arXiv preprint arXiv:2410.15672}, 2024.

\bibitem{barrata2023dolfinx}
Igor~A Barrata, Joseph~P Dean, J{\o}rgen~S Dokken, Michal Habera, Jack HALE,
  Chris Richardson, Marie~E Rognes, Matthew~W Scroggs, Nathan Sime, and Garth~N
  Wells.
\newblock Dolfinx: The next generation fenics problem solving environment.
\newblock 2023.

\bibitem{bemporad2002master}
Alberto Bemporad, Alessandro Giua, and Carla Seatzu.
\newblock A master-slave algorithm for the optimal control of continuous-time
  switched affine systems.
\newblock In {\em Proceedings of the 41st IEEE Conference on Decision and
  Control, 2002.}, volume~2, pages 1976--1981. IEEE, 2002.

\bibitem{bestehorn2021mixed}
Felix Bestehorn, Christoph Hansknecht, Christian Kirches, and Paul Manns.
\newblock Mixed-integer optimal control problems with switching costs: a
  shortest path approach.
\newblock {\em Mathematical Programming}, 188(2):621--652, 2021.

\bibitem{buchheim2024parabolic}
Christoph Buchheim, Alexandra Gr{\"u}tering, and Christian Meyer.
\newblock Parabolic optimal control problems with combinatorial switching
  constraints--part iii: Branch-and-bound algorithm.
\newblock {\em arXiv preprint arXiv:2401.10018}, 2024.

\bibitem{burger2023gauss}
Adrian B{\"u}rger, Clemens Zeile, Angelika Altmann-Dieses, Sebastian Sager, and
  Moritz Diehl.
\newblock A gauss--newton-based decomposition algorithm for nonlinear
  mixed-integer optimal control problems.
\newblock {\em Automatica}, 152:110967, 2023.

\bibitem{conn2000trust}
Andrew~R Conn, Nicholas~IM Gould, and Philippe~L Toint.
\newblock {\em Trust region methods}.
\newblock SIAM, 2000.

\bibitem{de2019mixed}
Alberto De~Marchi.
\newblock On the mixed-integer linear-quadratic optimal control with switching
  cost.
\newblock {\em IEEE Control Systems Letters}, 3(4):990--995, 2019.

\bibitem{fletcher2002global}
Roger Fletcher, Sven Leyffer, and Philippe~L Toint.
\newblock On the global convergence of a filter--sqp algorithm.
\newblock {\em SIAM Journal on Optimization}, 13(1):44--59, 2002.

\bibitem{gau2017novel}
Sebastian Gau, Thomas Leifeld, and Ping Zhang.
\newblock A novel and fast mpc based control strategy for switched linear
  systems including soft switching cost.
\newblock In {\em 2017 IEEE 56th Annual Conference on Decision and Control
  (CDC)}, pages 6513--6518. IEEE, 2017.

\bibitem{hahn2023binary}
Mirko Hahn, Sven Leyffer, and Sebastian Sager.
\newblock Binary optimal control by trust-region steepest descent.
\newblock {\em Mathematical Programming}, 197(1):147--190, 2023.

\bibitem{kirches2019generation}
Christian Kirches, Ekaterina~A Kostina, Andreas Meyer, and Matthias
  Schl{\"o}der.
\newblock Generation of optimal walking-like motions using dynamic models with
  switches, switch costs, and state jumps.
\newblock In {\em 2019 IEEE 58th Conference on Decision and Control (CDC)},
  pages 1538--1543. IEEE, 2019.

\bibitem{kirches2022sequential}
Christian Kirches, Jeffrey Larson, Sven Leyffer, and Paul Manns.
\newblock Sequential linearization method for bound-constrained mathematical
  programs with complementarity constraints.
\newblock {\em SIAM Journal on Optimization}, 32(1):75--99, 2022.

\bibitem{larsson1994class}
Torbj{\"o}rn Larsson and Michael Patriksson.
\newblock A class of gap functions for variational inequalities.
\newblock {\em Mathematical Programming}, 64:53--79, 1994.

\bibitem{leyffer2022sequential}
Sven Leyffer and Paul Manns.
\newblock Sequential linear integer programming for integer optimal control
  with total variation regularization.
\newblock {\em ESAIM: Control, Optimisation and Calculus of Variations}, 28:66,
  2022.

\bibitem{manns2023convergence}
Paul Manns, Mirko Hahn, Christian Kirches, Sven Leyffer, and Sebastian Sager.
\newblock On convergence of binary trust-region steepest descent.
\newblock {\em Journal of Nonsmooth Analysis and Optimization}, 4(3), 2023.

\bibitem{manns2023on}
Paul Manns and Annika Schiemann.
\newblock On integer optimal control with total variation regularization on
  multidimensional domains.
\newblock {\em SIAM Journal on Control and Optimization}, 61(6):3415--3441,
  2023.

\bibitem{marko2023integer}
Jonas Marko and Gerd Wachsmuth.
\newblock Integer optimal control problems with total variation regularization:
  Optimality conditions and fast solution of subproblems.
\newblock {\em ESAIM: Control, Optimisation and Calculus of Variations}, 29:81,
  2023.

\bibitem{sager2021mixed}
Sebastian Sager and Clemens Zeile.
\newblock On mixed-integer optimal control with constrained total variation of
  the integer control.
\newblock {\em Computational Optimization and Applications}, 78(2):575--623,
  2021.

\bibitem{severitt2023efficient}
Marvin Severitt and Paul Manns.
\newblock Efficient solution of discrete subproblems arising in integer optimal
  control with total variation regularization.
\newblock {\em INFORMS Journal on Computing}, 35(4):869--885, 2023.

\bibitem{sharma2021inversion}
Meenarli Sharma, Mirko Hahn, Sven Leyffer, Lars Ruthotto, and Bart van
  Bloemen~Waanders.
\newblock Inversion of convection--diffusion equation with discrete sources.
\newblock {\em Optimization and Engineering}, 22:1419--1457, 2021.

\bibitem{terpin2022trust}
Antonio Terpin, Nicolas Lanzetti, Batuhan Yardim, Florian Dorfler, and Giorgia
  Ramponi.
\newblock Trust region policy optimization with optimal transport
  discrepancies: Duality and algorithm for continuous actions.
\newblock {\em Advances in Neural Information Processing Systems},
  35:19786--19797, 2022.

\bibitem{toint1988global}
Philippe~L Toint.
\newblock Global convergence of a class of trust-region methods for nonconvex
  minimization in {H}ilbert space.
\newblock {\em IMA Journal of Numerical Analysis}, 8(2):231--252, 1988.

\bibitem{toint1997non}
Philippe~L. Toint.
\newblock Non-monotone trust-region algorithms for nonlinear optimization
  subject to convex constraints.
\newblock {\em Mathematical Programming}, 77:69--94, 1997.

\bibitem{van1987simulated}
Peter~JM Van~Laarhoven and Emile~HL Aarts.
\newblock {\em Simulated Annealing: Theory and Applications}.
\newblock Springer, 1987.

\end{thebibliography}

\end{document}